\documentclass[12pt]{amsart}
\usepackage[english]{babel}
\usepackage{amsmath, amsthm, amssymb}
\usepackage{enumitem}
\setlist[enumerate, 1]{label={\textup{(\roman*)}}}
\setlist[enumerate, 2]{label={\textup{(\alph*)}}}
\setlist{topsep=\smallskipamount, leftmargin=1.5em, listparindent=1em}
\usepackage{geometry}
\usepackage{hyperref}
\usepackage{tikz-cd}
\usetikzlibrary{arrows.meta}
\usepackage{algorithm}
\usepackage{array}

\theoremstyle{plain}
\newtheorem{theorem}{Theorem}[section]
\newtheorem{lemma}[theorem]{Lemma}
\newtheorem{prop}[theorem]{Proposition}

\newtheorem*{citedtheorem}{Theorem}

\theoremstyle{definition}
\newtheorem{definition}[theorem]{Definition}
\newtheorem{example}[theorem]{Example}

\theoremstyle{remark}
\newtheorem{remark}[theorem]{Remark}

\numberwithin{equation}{section}

\newcommand\uv[1]{``#1''}

\let\zet\Z
\let\kve\Q
\let\ce\C
\renewcommand{\P}{\mathbb{P}}\let\pe\P
\let\ef\F
\newcommand\OO{\mathcal{O}}

\let\epsilon\varepsilon

\newcommand\zav[1]{\left(#1\right)}
\newcommand\set[1]{\left\{#1\right\}}
\newcommand\abs[1]{\left|#1\right|}
\newcommand\gener[1]{\left\langle#1\right\rangle}
\DeclareMathOperator{\codim}{codim}
\DeclareMathOperator{\Sym}{Sym}
\newcommand\lows{\underline s}
\newcommand\ups{\overline s}\newcommand\upr{\overline r}
\newcommand\floor[1]{\left\lfloor#1\right\rfloor}
\newcommand\ceil[1]{\left\lceil#1\right\rceil}
\DeclareMathOperator{\rank}{rank}
\newcommand\restr[2]{\left.#1\right|_{#2}}
\DeclareMathOperator{\vdim}{vdim}

\newcommand\twovenn[5]{%
    \begin{tikzpicture}[scale=.975, line width=.6pt]
        \draw (-2.25,-1.5) rectangle (2.25,1.5);
        \path (-2.25,1.5) node[above right]{$#1$};
        \draw (-.6,.1) circle (1.1cm);
        \draw (.6,.1) circle (1.1cm);
        \path
            (2.2,-1.45) node[above left]{$#2$}
            (-1.075,.1) node{\hbox to2em{\hfil$#3$\hfil}}
            (1.075,.1) node{\hbox to2em{\hfil$#4$\hfil}}
            (0,.1) node{$#5$}
        ;
    \end{tikzpicture}
}

\newcommand\threevenn[2]{%
    \begin{tikzpicture}[scale=.975]
        #2
        \draw (-2.25,-2.35) rectangle (2.25,1.5);
        \path (-2.25,1.5) node[above right]{$#1$};
        \draw (-.6,.1) circle (1.1cm);
        \draw (.6,.1) circle (1.1cm);
        \draw[dashed] (-.6,.1)++(300:1.2cm) circle(1.1cm);
    \end{tikzpicture}
}
\newcommand\tmpangle{56.944268849146}
\newcommand\tmpcolor{red!25!white}

\newcommand\arxiv[1]{\href{https://arxiv.org/abs/#1}{arXiv:#1}}

\newcommand\tinylabel[1]{\hbox spread-6pt{\tiny#1}}

\newbox\btwobox\setbox\btwobox=\hbox{$B_2$}
\newcommand\familynameCommuting[1]{\hbox to\wd\btwobox{\hss$#1\vphantom{B_2}$\hss}}
\newdimen\containmentwd\containmentwd=.4\textwidth

\newcommand\familyname[2]{\hyperlink{def:#1#2}{$#1_#2$}}
\newcommand\famitem[2]{\item[{($#1_#2$)}]\hypertarget{def:#1#2}}

\newcommand\hatname[2]{\hyperlink{def:#1#2hat}{$\hat #1_#2$}}
\newcommand\hatitem[2]{\item[{($\hat #1_#2$)}]\hypertarget{def:#1#2hat}}

\newcommand\gitadress{https://github.com/tatarkajenej/SegVer-O-1-2-}
\newcommand\giturl{\url{\gitadress}}

\begin{document}

\title[Secants of Segre-Veronese $\mathbb P^{m}\times\mathbb P^{n}$ with $\mathcal O(1,2)$ are non-defective for $n\gg m^3$, $m\geq3$]{Secant varieties of Segre-Veronese varieties $\mathbb P^m\times\mathbb P^n$ embedded by $\mathcal O(1,2)$ are non-defective for $n\gg m^3$, $m\geq3$}

%% commented out for anonymity
\author{Matěj Doležálek}
\address{Fachbereich Mathematik und Statistik, Universität Konstanz, Konstanz, Germany}
\email{matej.dolezalek@uni-konstanz.de}
\author{Nikhil Ken}
\address{Dipartmento di Matematica e Informatica ``Ulisse Dini'',  Università di Firenze, Firenze, Italy}
\email{nikhil.ken@unifi.it}

\subjclass[2020]{Primary 14N07;
    Secondary 14N20, 14Q15, 15A69}
\keywords{Secant varieties, Segre-Veronese varieties, non-defectivity}

\thanks{This work has been supported by European Union’s HORIZON–MSCA-2023-DN-JD programme under the Horizon Europe (HORIZON) Marie Skłodowska-Curie Actions, grant agreement 101120296 (TENORS)}

\begin{abstract}
We prove that for any $m\geq3$, $n\gg m^3$, all secant varieties of the Segre-Veronese variety $\mathbb P^m\times\mathbb P^n$ have the expected dimension. This was already proved by Abo and Brambilla in the subabundant case, hence we focus on the superabundant case. We generalize an approach due to Brambilla and Ottaviani into a construction we call the \emph{inductant}. With a combinatorial investigation of these constructions, the proof of non-defectivity reduces to checking a finite collection of base cases, which we verify using a computer-assisted proof.
\end{abstract}

\maketitle

\tableofcontents

\section{Introduction}

The study of secant varieties is a classical problem of algebraic geometry. When restricted to varieties such as the Segre or Veronese varieties, it also presents a geometric formulation of questions concerning border rank or symmetric border rank in tensor spaces. The famous theorem of Alexander and Hirschowitz \cite{alexander-hirschowitz95, alexander-hirschowitz97} solved completely the problem of determining the dimensions of secant varieties of Veronese varieties. It says that this dimension is always the so-called \emph{expected dimension}, except for an explicitly enumerated and relatively small set of defective cases. Brambilla and Ottaviani \cite{brambilla-ottaviani} later gave an alternative presentation, some of whose techniques motivate this article.

For Segre varieties or for Segre-Veronese varieties, complete solutions in the spirit of the Alexander-Hirschowitz theorem are currently not known, although many partial results are available. Strassen \cite{strassen69, strassen83} originally studied the problem of tensor rank (and border rank) in connection with complexity of matrix multiplication. Lickteig \cite{lickteig} provided further results on three-factor Segre varieties, notably in the case of $\pe^n\times\pe^n\times\pe^n$.
Abo, Ottaviani and Peterson \cite{abo-ottaviani-peterson} studied $(\pe^n)^k$ in greater generality using a general inductive technique; they also classified defective $s$-th secant varieties of Segre varieties for small $s$. Blomenhofer and Casarotti \cite{blomenhofer-casarotti} gave a unified treatment for Segre-Veronese varieties, Grassmannians and other varieties and proved non-defectivity of their $s$-th secant varieties for $s$ outside of a certain small interval.

In the case of Segre-Veronese varieties, the problem of defective secant varieties has also received much attention in recent years, especially in the two-factor case. Bernardi, Carlini and Catalisano \cite{bernardi-carlini-catalisano} studied $s$-th secant varieties of $\pe^n\times\pe^n$ embedded by $\OO(1,d)$, $d\geq3$ and proved non-defectivity for $s$ outside of an interval bounded by multiples of $n+1$ nearest to $\frac{(n+1)\binom{m+d}d}{m+n+1}$. In \cite{ballico-bernardi-catalisano}, Ballico, Bernardi and Catalisano completely characterized dimensions of secant varieties to $\pe^n\times\pe^1$ embedded by any $\OO(a,b)$.

In general, the trend with Segre-Veronese varieties seems to be that factors of high degree make the problem easier.
Most notably in this vein, Abo, Brambilla, Galuppi and Oneto \cite{abo-brambilla-galuppi-oneto} proved recently that Segre-Veronese varieties $\pe^{n_1}\times\cdots\times\pe^{n_k}$ embedded by $\OO(d_1,\dots,d_k)$ with $k\geq3$ and all $d_i\geq 3$ never have defective secant varieties. This result built on earlier works by Galuppi and Oneto \cite{galuppi-oneto} and by Ballico \cite{ballico}. Ballico \cite{ballico2024} proved further improvements on the results of Abo, Brambilla, Galuppi and Oneto which allow for some of the factors to be of degree $d_i=2$. All of these results rely on first proving non-defectivity in some concrete multidegrees -- such as $\OO(3,3)$, $\OO(3,4)$ and $\OO(4,4)$ in \cite{galuppi-oneto} --  and subsequently performing inductions on degrees and on the number of factors.

These inductive techniques leave the cases with small degrees as seemingly the most interesting. The most notable of these are $\OO(1,1,1)$ and $\OO(1,2)$, the latter of which is the focus of this article. In this domain, Cartwright, Erman and Oeding \cite{cartwright-erman-oeding} provided equations defining $s$-th secant varieties of $\pe^2\times\pe^n$ embedded by $\OO(1,2)$ for small $s$. The case of $\pe^2\times\pe^n$ is also notable because it contains an infinite family of examples of defective secant varieties \cite[Theorem 5.3]{abo-brambilla}.
Abo and Brambilla worked with a general $\pe^m\times\pe^n$, proving the following:

\begin{citedtheorem}[{\cite[Theorem 3.13]{abo-brambilla}}]
Let $m$, $n$ be positive integers with $n\geq m-2$. Then for any $s\leq \lows(m,n) := \floor{\frac{(m+1)(n-m+2)}{2}}$, the Segre-Veronese variety $\pe^m\times\pe^n$ embedded by $\OO(1,2)$ has a non-defective $s$-th secant variety.
\end{citedtheorem}

For $s < \frac{(m+1)\binom{n+2}2}{n+m+1}$, the secant variety cannot fill the ambient space $\pe^{(m+1)\binom{n+2}2-1}$ and hence the cases $s\leq \floor{\frac{(m+1)\binom{n+2}2}{n+m+1}}$ are called \emph{subabundant}. Since the bound $\floor{\frac{(m+1)(n-m+2)}{2}}$ in the result of Abo and Brambilla equals $\floor{\frac{(m+1)\binom{n+2}2}{n+m+1}}$ for $n$ large enough with respect to $m$, this settles all subabundant cases for sufficiently large $n$.
% For sufficiently large $n$, this bound equals $\floor{\frac{(m+1)\binom{n+2}2}{n+m+1}}$, making this result optimal in the sense of the so-called subabundant case, because one expects $s= \ceil{\frac{(m+1)\binom{n+2}2}{n+m+1}}$ to be the first $s$ for which the $s$-th secant variety fills the ambient space.
Abo \cite{abo} also investigated $\pe^n\times\pe^n$ and $\pe^{n+1}\times\pe^n$ with $\OO(1,2)$, proving that these are non-defective except for $\pe^4\times\pe^3$.

Our main result in this paper is the following:
\begin{theorem}
    \label{thrm:main}
    If $m\geq3$ and
    \[
        n\geq\begin{cases}
            m^3-m-3, & \text{if $m$ is even and $n$ is odd},\\
            \ceil{\frac{m^3-2m-5}2}, & \text{otherwise},
        \end{cases}
    \]
    then all secant varieties of the Segre-Veronese variety $\pe^{m}\times\pe^{n}$ embedded by $\OO(1,2)$ are non-defective.
\end{theorem}
We will prove this through Theorems~\ref{thrm:A0nice} and \ref{thrm:A0ugly}.
Our focus will be on the so-called \emph{superabundant} case, when $s\geq \ceil{\frac{(m+1)\binom{n+2}2}{n+m+1}}$,
in order to complement the subabundant result due to Abo and Brambilla. In particular, we will prove that under the bounds of Theorem~\ref{thrm:main}, the $s$-th secant variety for $s=\ups(m,n):=\lows(m,n)+1$ is superabundant and non-defective. This will then immediately imply the same also for $s\geq \ups(m,n)$, hence obtaining Theorem~\ref{thrm:main} when combined with the result of Abo and Brambilla.

Previously, there remained even for large $n$ an interval of values of $s$ for which the non-defectivity of the secant variety was not known. We are able to close this gap (for large $n$) by systematically handling the combinatorics of configurations that arise from specializing points to subspaces.

Note that cases $m=1$ and $m=2$ have already been completely solved (see \cite[Corollary 1.4]{abo-brambilla}), and leaving them out allows us to sidestep the defective cases with $m=2$ and $n$ odd.
Let us also note that our result is consistent with the list of defective cases which was conjectured to be exhaustive in \cite[Section 5]{abo-brambilla}.

To close out this introduction, we outline the structure of the paper. In Section~\ref{sec:preliminaries} we overview the basic setting of secant varieties and prepare some arithmetic related to secant varieties of Segre-Veronese varieties embedded by $\OO(1,2)$. In Section~\ref{sec:coord-configs}, we introduce the notions of \emph{coordinate configurations} and the \emph{inductant} construction. In Section~\ref{sec:strategies}, we apply these to prove Theorem~\ref{thrm:main}. Finally, in Section~\ref{sec:computations}, we describe the computations which were needed to supply the base cases of inductive proofs in Section~\ref{sec:strategies}; the source code and resulting certificates for these computations are available on GitHub at \cite{gitrepo}.

% commented out for anonymity
\section*{Acknowledgments}

This project started during the first author's secondment at the University of Florence.

We thank Giorgio Ottaviani for his careful reading and numerous comments on several drafts of this paper and Alessandra Bernardi, Maria Chiara Brambilla, Francesco Galuppi, Mateusz Michałek, Alessandro Oneto and Nick Vannieuwenhoven for helpful suggestions. Lastly, we thank Alberto Mancini for his help with running our computations.

\section{Preliminaries}
\label{sec:preliminaries}

\subsection{Segre-Veronese varieties and their secant varieties}

Throughout this paper, let $U$, $V$ be $(m+1)$- and $(n+1)$-dimensional $\ce$-vector spaces respectively and let $\pe^{m,n}=\pe(U)\times\pe(V)$. Let us fix coordinate systems $x_0,\dots,x_m$ on $U$ and $y_0,\dots,y_n$ on $V$. Further, let $X_{m,n}$ denote the Segre-Veronese variety $\pe^{m,n}$ embedded by $\OO(1,2)$ in $\pe^{(m+1)\binom{n+2}2-1}$.

The variety $\pe^{m,n}=\pe(U)\times\pe(V)$ is naturally equipped with the bigraded algebra $\Sym U\otimes\Sym V$, which we identify with the polynomial algebra $R=\ce[x_0,\dots,x_m,y_0,\dots,y_n]$. For any set $L\subseteq\pe^{m,n}$, we may consider its bihomogeneous ideal $\mathcal I_L$, which is generated by all bihomogeneous polynomials that vanish at all points $P\in L$. For any bihomogeneous ideal $\mathcal I\leq R$, we denote its bidegree $(a,b)$ homogeneous part by $\mathcal I(a,b)$. A~homogeneous $f\in R$ is said to be \emph{singular} at some $P\in \pe^{m,n}$ if all its partial derivatives vanish when evaluated at $P$.

In general, with a variety $X\subseteq\pe^N$, its $s$-th secant variety $\sigma_s(X)$ is defined as the Zariski closure of the union of all linear spans of $s$-tuples of points on $X$, that is
\[
    \sigma_s(X) := \overline{\bigcup_{P_1,\dots,P_s\in X}\langle P_1,\dots,P_s\rangle}.
\]
The expected dimension of $\sigma_s(X)$ is then $\min\{N, s (\dim X+1)-1\}$, we say $\sigma_s(X)$ is \emph{defective}, if it does not have the expected dimension, and \emph{non-defective} otherwise. In particular, we say $X$ itself is \emph{defective}, if $\sigma_s(X)$ is defective for some $s$, and \emph{non-defective}, if all $\sigma_s(X)$ are non-defective. Further, $\sigma_s(X)$ is said to be \emph{subabundant} (resp. \emph{superabundant}), if $s (\dim X+1)-1\leq N$ (resp. $s (\dim X+1)-1\geq N$). It is further said to be \emph{equiabundant}, if it is both sub- and superabundant.

One may investigate $\dim \sigma_s(X)$ by taking the tangent space at a generic point. By Terracini's lemma, the tangent space to $\sigma_s(X)$ at a generic point $Q$ in the span of $s$ generic points $P_1,\dots,P_s\in X$ equals the span of tangent spaces to $X$ at $P_1,\dots,P_s$. Returning to the Segre-Veronese variety $X = X_{m,n}$, we may use this view to reformulate the problem of finding $\dim\sigma_s(X_{m,n})$ as follows:
If $P_1,\dots,P_s$ are chosen generically on $\pe^{m,n}$ and $\mathcal I\leq \Sym U\otimes\Sym V = \ce[x_0,\dots,x_m,y_0,\dots,y_n]$ is the bihomogeneous ideal consisting of polynomials that are singular at each $P_i$, then the affine dimension of $\sigma_s(X_{m,n})$ equals
\[
    \codim\mathcal I(1,2) = (m+1)\binom{n+2}2 - \dim\mathcal I(1,2),
\]
where we consider the dimension of $\mathcal I(1,2)$ as a vector space over $\ce$.
Thus we may formulate our inquiries in terms of how many independent conditions does being singular at $s$ generic points of $\pe^{m,n}$ put on bidegree $(1,2)$ polynomials -- non-defectivity means $\dim\mathcal I(1,2) = \max\left\{0,\ (m+1)\binom{n+2}2 - s(m+n+1)\right\}$ in this language.

Abo and Brambilla \cite[Corollary 3.14]{abo-brambilla} proved that $\sigma_s(X_{m,n})$ is non-defective for any $s\leq \lows(m,n) := \floor{\frac{(m+1)(n-m+2)}2}$. Since this $\lows(m,n)$ equals $\floor{\frac{(m+1)\binom{n+2}2}{n+m+1}}$ for
\[
    n > \begin{cases}
        m^3-2m, & \text{if $m$ is even and $n$ is odd,}\\
        \frac{(m-2)(m+1)^2}2, & \text{otherwise,}
    \end{cases}
\]
this means that the subabundant secant varieties of $X_{m,n}$ are known to be non-defective under these conditions. For superabundant cases, Abo and Brambilla proved non-defectivity for $s$ satisfying asymptotically tight bounds, which are however strictly higher than $\ceil{\frac{(m+1)\binom{n+2}2}{n+m+1}}$ even for large $n$. Thus in this paper, our goal is to improve these non-defectivity results in the superabundant case. This will obtain non-defectivity of $X_{m,n}$ for $n\gg m^3$ when combined with the results on subabundant cases due to Abo and Brambilla.

\subsection{Arithmetic considerations}

Let
\begin{align*}
    \ups(m,n) &:= \floor{\frac{(m+1)(n-m+2)}2}+1, & \upr(m,n) &:= \ceil{\frac{(m+1)\binom{n+2}2}{n+m+1}}.
\end{align*}
We will observe in Lemma~\ref{lem:arithmetic} that these two functions agree once $n$ is sufficiently large with respect to $m$. By definition, $s=\upr(m,n)$ is always the smallest value of $s$ for which the $s$-th secant variety of $X_{m,n}$ is superabundant. The point of considering $\ups(m,n)$ is that we express $\upr(m,n)$ as a case-wise polynomial function, although this expression holds only for large values of $n$.

Note that in $\ups(m,n)$, we may split cases based on parity of $m$ and $n$ to arrive at the following case-wise definition:
\[
    \ups(m,n) = \begin{cases}
        \frac{(m+1)(n-m+2)+1}2, & \text{if $m$ is even and $n$ is odd,}\\
        \frac{(m+1)(n-m+2)+2}2, & \text{otherwise.}
    \end{cases}
\]
Due to this arithmetic behavior, a splitting of cases into \uv{$m$ even, $n$ odd} and \uv{the rest} will be common throughout this paper, so we will use the following shorthand:
\begin{definition}
    Let us say a pair of non-negative integers $(m,n)$ is
    \begin{itemize}
        \item \emph{ugly}, if $m$ is even and $n$ is odd,
        \item \emph{nice}, otherwise.
    \end{itemize}
\end{definition}

\begin{lemma}
    \label{lem:arithmetic}
    If
    \begin{equation}
        \label{eq:arithmetic-n-bound}
        n \geq \begin{cases}
            3\binom{m+1}3-m-1, & \text{if $(m,n)$ is nice,}\\
            6\binom{m+1}3-m-1, & \text{if $(m,n)$ is ugly,}
        \end{cases}
    \end{equation}
    then $\upr(m,n)=\ups(m,n)$ and
    \begin{equation}
        \label{eq:arithmetic-remainder}
        (m+1)\binom{n+2}2 - (n+m+1)\upr(m,n) = \begin{cases}
            3\binom{m+1}3-n-m-1, & \text{if $(m,n)$ is nice,}\\
            3\binom{m+1}3 - \frac{n+m+1}2, & \text{if $(m,n)$ is ugly.}
        \end{cases}
    \end{equation}
\end{lemma}
\begin{proof}
    In $\kve[m,n]$, we may observe
    \[
        (m+1)\binom{n+2}2 = (n+m+1)\cdot\frac{(m+1)(n-m+2)}2 + 3\binom{m+1}3
    \]
    and hence also
    \[
        (m+1)\binom{n+2}2 = (n+m+1)\cdot\floor{\frac{(m+1)(n-m+2)}2} + \begin{cases}
            3\binom{m+1}3, & \text{if $(m,n)$ is nice,}\\
            3\binom{m+1}3 + \frac{n+m+1}2, & \text{if $(m,n)$ is ugly.}
        \end{cases}
    \]
    Note that if we prove $\upr(m,n)=\ups(m,n)$, then this previous equality already rearranges to \eqref{eq:arithmetic-remainder}.
    Subsequently, once $n$ is large enough so that
    \[
        n+m+1 \geq \begin{cases}
            3\binom{m+1}3, & \text{if $(m,n)$ is nice,}\\
            3\binom{m+1}3 - \frac{n+m+1}2, & \text{if $(m,n)$ is ugly,}
        \end{cases}
    \]
    we observe $\upr(m,n) = \floor{\frac{(m+1)(n-m+2)}2} + 1 = \ups(m,n)$. This bound rearranges to \eqref{eq:arithmetic-n-bound}.
\end{proof}

\section{Families of coordinate configurations}
\label{sec:coord-configs}

In this section, we introduce the inductant construction which we will use to present our inductive strategy for the proof of non-defectivity of the superabundant cases of secant varieties of Segre-Veronese varieties. We formulate this in bidegree $(1,2)$ only for the sake of simplicity, but we hope the ways to generalize to other bidegrees or even to Segre-Veronese varieties with more factors can be easily inferred.

In inductive proofs of non-defectivity of secant varieties, it is common to specialize double points on hyperplanes. This approach does not work well for bidegree $(1,2)$ however because of the many defectivities in bidegrees $(1,1)$ and $(0,2)$. This is similar to the cubic case of the Alexander-Hirschowitz theorem, where Brambilla and Ottaviani \cite{brambilla-ottaviani} were able to overcome this issue by specializing points to higher-codimensional subspaces. Such a specialization may be repeated several times, which leads one to consider collections of subvarieties stemming from such subspaces alongside double points chosen with various constraints.

\begin{definition}
    Let us say a subspace $U'\leq U$ is a \emph{coordinate subspace}, if it is defined by equations $x_i=0$, $i\in I$ for some $I\subseteq\set{0,1,\dots,m}$. Analogously, let us say a $V'\leq V$ is a coordinate subspace if it is given by $y_j=0$, $j\in J$ where $J\subseteq\set{0,1,\dots,n}$. Let us say an algebraic subvariety $L\subseteq \pe^{m,n}$ is a \emph{coordinate subvariety}, if it is of the form $\pe(U')\times\pe(V')$ for coordinate subspaces $U'$, $V'$ of $U$, $V$ respectively. We then also write $L\simeq \pe^{\dim U'-1,\dim V'-1}$.
\end{definition}

\begin{definition}
    \label{def:coord-config}
    For integers $k,m,n\geq0$, let three $2^k$-tuples of nonnegative integers $p_I$, $\tilde u_I$, $\tilde v_I$ indexed by sets $I\subseteq\set{1,\dots,k}$ and satisfying the conditions
    \begin{align}
        \label{eq:coordsubvar-nonempty}
        &\forall t\in\set{1,\dots,k}:\quad \sum_{J\not\ni t}\tilde u_J > 0,\quad \sum_{J\not\ni t}\tilde v_J>0,\\
        \label{eq:coordconfig-sanity}
        &\forall I:\quad \text{if $p_I>0$,}\quad \text{then }\displaystyle\sum_{J\cap I=\emptyset} \tilde u_J>0\text{ and } \sum_{J\cap I=\emptyset}\tilde v_J>0
    \end{align}
    and $\sum \tilde u_I = m+1$, $\sum\tilde v_I = n+1$ define a \emph{coordinate configuration} $Z=Z(\set{p_I},\set{\tilde u_I},\set{\tilde v_I})$ as follows:

    We partition the coordinates $x_0,\dots,x_m$ into $2^k$ sets of sizes $\tilde u_I$ and for $t=1,\dots,k$, let the coordinate subspace $U_t\leq U$ be given by equations $x_i=0$ for all $x_i$ from the sets corresponding to $I\ni t$. Analogously, we construct the coordinate subspaces $V_t\leq V$, $t=1,\dots,k$ and then set coordinate subvarieties $L_t := \pe(U_t)\times\pe(V_t)\subseteq\pe^{m,n}$.

    Then, we choose $p_I$ points on $\bigcap_{t\in I}L_t$ (we interpret the empty intersection as $\pe^{m,n}$) for each $I$, and we make this choice in a general fashion.

    The coordinate configuration $Z$ then consists of $L_1,\dots,L_k$ and double points supported in the points chosen above.
\end{definition}
Let us comment and expand on parts of the Definition~\ref{def:coord-config}; afterward, we will illustrate both the definition and our comments in Example~\ref{ex:basic}.

First, let us note that choosing different partitions of $x_0,\dots,x_m$ resp. $y_0,\dots,y_n$ into sets of sizes $\tilde u_I$ resp. $\tilde v_I$ only changes the resulting configuration $Z$ by a permutation of coordinates, i.e. by automorphisms of the underlying vector spaces -- therefore, we will disregard these differences. Further, since $U_t$ satisfies exactly those coordinate equations that correspond to index sets $I\ni t$, it will have $\codim U_t=\sum_{I\ni t} \tilde u_I$, and analogously $\codim V_t = \sum_{I\ni t}\tilde v_I$. Correspondingly then, the sums $\sum_{J\not\ni t}\tilde u_J$, resp. $\sum_{J\not\ni t}\tilde v_J$ give $\dim U_t$, resp. $\dim V_t$, so the condition \eqref{eq:coordsubvar-nonempty} says that $L_t = \pe(U_t)\times\pe(V_t)$ should be non-empty. We also see that $\codim L_t = \sum_{I\ni t}(\tilde u_I+\tilde v_I)$.

Further, in this view, $\tilde u_\emptyset$ and $\tilde v_\emptyset$ give the numbers of coordinate equations that are not satisfied by any of the respective coordinate subspaces.

Next, we comment on the conditions \eqref{eq:coordconfig-sanity}.
Notice that $\sum_{J\cap I=\emptyset} \tilde u_J$ equals $\dim\bigcap_{t\in I}U_t$: indeed, this intersection satisfies exactly those equations that are satisfied by at least one $U_t$, $t\in I$, and in turn $U_t$ satisfies exactly those coordinate equations that correspond to an index set $J\ni t$. Thus
\[
    \codim \bigcap_{t\in I}U_t = \sum_{J\cap I\neq \emptyset}\tilde u_J,
\]
and subtracting from $m+1 = \sum \tilde u_J$ yields $\dim\bigcap_{t\in I}U_t = \sum_{J\cap I=\emptyset}\tilde u_J$. Analogously, $\sum_{J\cap I=\emptyset}\tilde v_J = \dim\bigcap_{t\in I}V_t$. Hence, the conditions \eqref{eq:coordconfig-sanity} merely state that if some points are to be constrained to $\bigcap_{t\in I}L_t$, then this intersection ought to be non-empty. Note that in verifying the conditions~\eqref{eq:coordconfig-sanity}, it is clearly enough to check sets $I$ that are inclusion-maximal among those with $p_I>0$.

Further, the discussion of the previous paragraph shows that whenever the intersection $\bigcap_{t\in I}L_t$ is non-empty, its codimension is
\[
    \codim \bigcap_{t\in I}L_t = \sum_{J\cap I\neq\emptyset}(\tilde u_J+\tilde v_J).
\]

Whenever we have a coordinate configuration $Z=Z(\set{p_I},\set{\tilde u_I},\set{\tilde v_I})$ defined as above, it will be useful for us to denote codimensions $u_I:=\codim \sum_{t\in I}U_t$, $v_I:=\codim\sum_{t\in I}V_t$, because they will come up shortly in Definition~\ref{def:virtdim}. A~sum of coordinate subspaces is the coordinate subspace given by coordinate equations common to all the summands, and in turn, the codimension of a coordinate subspace is the number of coordinate equations it satisfies. Therefore, $u_I$ is the number of coordinate equations satisfied simultaneously by all $U_t$, $t\in I$. With this and the analogous statement for $v_I$, we obtain
\[
    u_I=\sum_{I\subseteq J\subseteq\set{1,\dots,k}}\tilde u_{J},\qquad v_I=\sum_{I\subseteq J\subseteq\set{1,\dots,k}}\tilde v_{J}
\]
by the construction described in Definition~\ref{def:coord-config}, because the coordinate equations corresponding to $J$ are common to all $U_t$, $t\in I$ if and only if $I\subseteq J$.
Note that the tuples $\set{u_I}$, $\set{v_I}$ could also be used to determine $Z$ instead of $\set{\tilde u_I}$, $\set{\tilde v_I}$, since the former can reconstruct the latter through the inclusion-exclusion formulae
\[
    \tilde u_I=\sum_{I\subseteq J\subseteq \set{1,\dots,k}} (-1)^{\abs J - \abs I} u_J,\qquad \tilde v_I=\sum_{I\subseteq J\subseteq \set{1,\dots,k}} (-1)^{\abs J - \abs I} v_J.
\]
Most properties of $Z$ are more clearly seen from $\tilde u_I$, $\tilde v_I$ than from $u_I$, $v_I$ though, so we will only use $u_I$, $v_I$ as quantities derived from $\tilde u_I$, $\tilde v_I$. Note that non-negativity of all $\tilde u_I$ and $\tilde v_I$ implies $u_I\geq u_J$, $v_I\geq v_J$ whenever $I\subseteq J$. Note also that always $u_\emptyset = m+1$, $v_\emptyset=n+1$.

\begin{example}
\label{ex:basic}
For integers $m\geq1$, $n\geq3$, let a coordinate configuration $Z$ in $\pe^{m,n}$ with $k=2$ subvarieties be given by
\begin{align*}
    p_\emptyset &= p_{\set{1,2}} = 0, & p_{\set1} &= p_{\set2} = m+1,\\
    \tilde u_\emptyset &= m+1, & \tilde u_{\set1} &= \tilde u_{\set2} = \tilde u_{\set{1,2}} = 0,\\
    \tilde v_\emptyset &= n-3, & \tilde v_{\set1} &= \tilde v_{\set2} = 2,\\
    \tilde v_{\set{1,2}} &= 0.
\end{align*}
We find it useful, at least with configurations with a small number of subvarieties, to visualize these parameters using Venn diagrams (see Figure~\ref{fig:venn-example}).

\begin{figure}[tb]
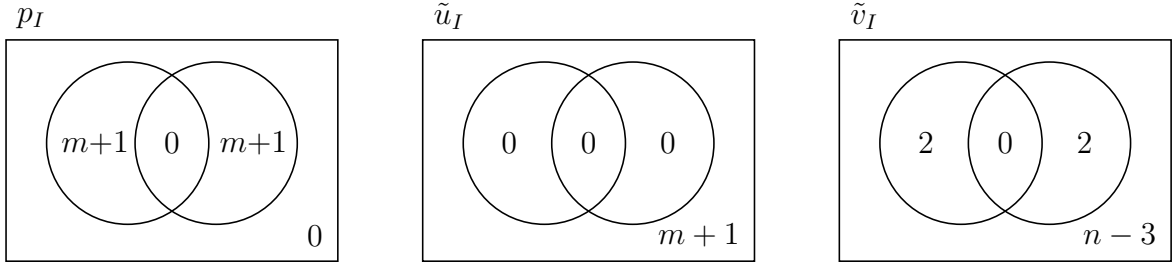

    \centering
    \twovenn{p_I}0{m+1}{m+1}0
    \qquad
    \twovenn{\tilde u_I}{m+1}000
    \qquad
    \twovenn{\tilde v_I}{n-3}220
    \caption{The parameters specifying the configuration $Z$ of Example~\ref{ex:basic}, visualized in Venn diagrams.}
    \label{fig:venn-example}
\end{figure}

These parameters mean we have two coordinate subvarieties $L_1$ and $L_2$. Neither of them is given by any $x_i=0$ coordinate equations, instead each has two equations in coordinates $y_j$ and these are not shared, so we may say without loss of generality that $L_1$ is given by $y_0=y_1=0$ and $L_2$ by $y_2=y_3=0$. In total, there are $2m+2$ double points in the configuration, $m+1$ constrained to $L_1$ and $m+1$ to $L_2$. Note that conditions \eqref{eq:coordsubvar-nonempty} and $\eqref{eq:coordconfig-sanity}$ are satisfied. Notice that $\sum_{J\cap \set{1,2}=\emptyset}\tilde v_j = \tilde v_\emptyset = n-3$ might be zero e.g. when $m=0$, $n=3$ (which corresponds to $L_1$ and $L_2$ having an empty intersection -- two skew lines in $\pe^3$), but this does not contradict~\eqref{eq:coordconfig-sanity}, since $p_{\set{1,2}}$ is zero.

Lastly, the sum codimensions $u_I$, $v_I$ are
\begin{align*}
    u_\emptyset &= m+1, & u_{\set1}&=u_{\set2} = u_{\set{1,2}}=0,\\
    v_\emptyset &= n+1, & v_{\set1}&=v_{\set2} = 2,\\
    v_{\set{1,2}} &= 0.
\end{align*}
\end{example}

\begin{definition}
    \label{def:virtdim}
    Let $Z=Z(\set{p_I},\set{\tilde u_I},\set{\tilde v_I})$ be a coordinate configuration in $\pe^{m,n}$ with $k$ subvarieties $L_1,\dots,L_k$. Denoting $L:=\bigcup_{t\in\set{1,\dots,k}}L_t$, we then define the \emph{virtual dimension (in $\OO(1,2)$)} of $Z$ as
    \begin{align*}
        \vdim Z := \dim\mathcal I_{L}(1,2) - \sum_{I\subseteq\set{1,\dots,k}} p_I\cdot \min\{u_I+v_I, m+n+1\}.
    \end{align*}
    We say $Z$ is \emph{superabundant} (resp. \emph{subabundant}) if $\vdim Z$ is non-positive (resp. non-negative), and we say $Z$ is \emph{equiabundant} if it is both superabundant and subabundant, i.e. if $\vdim Z=0$.

    If we denote by $P_1,\dots,P_s$ the double points that $Z$ has, we further define the \emph{ideal of $Z$} to be the bihomogeneous ideal $\mathcal I_Z$ that is generated by those bihomogeneous $f\in \mathcal I_L$ which are singular at all $P_i$.
    We then say $Z$ is \emph{non-defective (in $\OO(1,2)$)}, if
    \begin{equation}
        \label{eq:expected-dim}
        \dim \mathcal I_Z(1,2) = \max\set{0, \vdim Z};
    \end{equation}
    if not, we say $Z$ is \emph{defective (in $\OO(1,2)$)}.
\end{definition}

Note that due to conditions \eqref{eq:coordconfig-sanity}, if $p_I>0$, we have $\sum_{J\cap I=\emptyset}\tilde u_J>0$. For any non-empty $I$, the set of $J$'s with $J\cap I=\emptyset$ is then disjoint from those with $J\supseteq I$, so we obtain
\[
    u_I = \sum_{J\supseteq I}\tilde u_J \leq \zav{\sum_{J} \tilde u_J} - \zav{\sum_{J\cap I=\emptyset}\tilde u_J} < \sum_{J}\tilde u_J = m+1,
\]
i.e. $u_I\leq m$. Analogously, $v_I\leq n$.
Hence, whenever $p_I>0$, we may write
\begin{equation}
    \label{eq:minimum-simplifies}
    \min\{u_I+v_I, m+n+1\} = \begin{cases}
        m+n+1, & \text{if $I=\emptyset$,}\\
        u_I+v_I, &\text{otherwise,}
    \end{cases}
\end{equation}
and thus the definition of $\vdim Z$ could also be written as
\[
    \vdim Z = \dim\mathcal I_{L}(1,2) - p_{\emptyset}\cdot(m+n+1) - \sum_{\emptyset\neq I\subseteq\set{1,\dots,k}} p_I\cdot (u_I+v_I).
\]

Further, it is useful to note that when $L_t$ is a coordinate subvariety defined by equations $x_i=0$, $i\in I$ and $y_j=0$, $j\in J$, then its ideal is just the monomial ideal $\mathcal I_L = (x_i,y_j\mid i\in I, j\in J)$. Consequently, the ideal $\mathcal I_L$ of $L=\bigcup L_t$ may be computed as the intersection of the individual $\mathcal I_{L_t}$.

The right-hand side of \eqref{eq:expected-dim} is called the \emph{expected dimension} of $\mathcal I_Z(1,2)$ and it is unconditionally a lower bound for it, hence non-de\-fec\-ti\-vi\-ty is merely equivalent to an equality occurring in this lower bound. The motivation for summing the expression $u_I+v_I$ (subject to a minimum with $m+n+1$) is the following:

Suppose first that a double point $P$ is constrained to a coordinate subvariety $L_t$ and consider a bidegree $(1,2)$ monomial $f$. Let us abuse notation slightly by writing $\frac{\partial f}{\partial z}(P)$ for the evaluation of $\frac{\partial f}{\partial z}$ at the support of $P$.  If $f$ already vanishes on $L_t$ and $z$ is a coordinate that does not vanish on $L$, then $\frac{\partial f}{\partial z}(P)$ is always zero, because $\frac{\partial f}{\partial z}$ must still be divisible by a variable that does vanish on $L_t$. Hence the condition $\frac{\partial f}{\partial z}(P)=0$ is redundant. This leaves as relevant conditions, which we expect to decrease the dimension of $\mathcal I_Z(1,2)$, only those stemming from coordinates $z$ which do vanish on $L_t$, and there is exactly $u_{\set t}+v_{\set t}$ of those. More generally, if $P$ is instead constrained to some $\bigcup_{t\in I}L_t$ and $f$ already vanishes on each $L_t$, then $\frac{\partial f}{\partial z}(P)$ may be non-zero only if $z$ vanishes on all $L_t$, $t\in I$. Again, there are exactly $u_I+v_I$ such coordinates $z$, which we then expect to put non-trivial conditions on $f$.

Further, we may notice that some of the expressions $u_I+v_I$ in the definition of virtual dimensions might be zero -- if this is the case, let us say the intersection $\bigcap_{t\in I}L_t$ is \emph{irrelevant} and all of the double points constrained to it are \emph{irrelevant points}. As the following lemma will illustrate, we call them irrelevant, because they contribute nothing to the virtual dimension of $Z$ nor to the actual dimension of $\mathcal I_Z(1,2)$:
\begin{lemma}
    Let $Z$ be a coordinate configuration and let $Z'$ be the coordinate configuration obtained by removing all irrelevant points from $Z$. Then
    $\vdim Z' = \vdim Z$ and $\mathcal I_{Z'} = \mathcal I_Z$.
\end{lemma}
\begin{proof}
    By definition, irrelevant points contribute nothing to the sum in $\vdim Z$, hence their omission leaves the virtual dimension unchanged.

    For the ideals, it suffices to show that for any monomial $f\in\mathcal I_{\bigcap_{t\in\set{1,\dots,k}}L_t}$, all partial derivatives of $f$ vanish at any irrelevant point. Let $z$ be any of the variables $x_0,\dots,x_m,y_0,\dots,y_n$. If $z$ does not divide $f$, there is nothing to prove. If it does, we use that fact that $u_I+v_I=0$, which means that $z=0$ is not shared as an equation by all $L_t$, $t\in I$, so let us say some $L_j$ lacks it. Therefore, $f$ must be divisible by some other variable $w$ whose equation $w=0$ is one of the ones that define $L_j$. Then $w$ still divides $\frac{\partial f}{\partial z}$, hence $\frac{\partial f}{\partial z}(P) = 0$ for any $P\in L_j\supseteq \bigcap_{t\in\set{1,\dots,k}}L_t$.
\end{proof}
With the Lemma, we are justified in disregarding all irrelevant points when considering questions of abundancy or defectivity.

\begin{example}
    \label{ex:defect}
    Let us investigate the virtual dimension of the coordinate configuration $Z$ from Example~\ref{ex:basic}. First, the ideal $\mathcal I_L(1,2)$. The two subvarieties of $Z$ have ideals $(y_0,y_1)$ and $(y_2,y_3)$ respectively. These are prime ideals and their intersection coincides with their product $(y_0,y_1)(y_2,y_3)=(y_0y_2,y_0y_3,y_1y_2,y_1y_3)$. Thus the bidegree $(1,2)$ part of $\mathcal I_L$ will be spanned by monomials $x_iy_{j_1}y_{j_2}$ with $i$ arbitrary and $j_1\in\set{0,1}$, $j_2\in\set{2,3}$, which gives $\dim\mathcal I_L(1,2)=4(m+1)$.

    In the portion of $\vdim Z$ stemming from points, we previously noted that $u_I=2$ and $v_I=0$ for both $I=\set1$ and $I=\set2$ (and no other $I$ have $p_I>0$), so in total, we arrive at
    \[
        \vdim Z = 4(m+1) - (m+1)\cdot 2 - (m+1)\cdot2= 0,
    \]
    making $Z$ an equiabundant coordinate configuration. To prove its non-defectivity would amount to showing that $\mathcal I_Z(1,2)$ vanishes. We will not prove this, though it is proved in \cite[Proposition 3.11]{abo-brambilla} for all $m\geq1$, $n\geq3$, up to translating into the terminology of coordinate configurations.

    Lastly, we note that the intersection $L_1\cap L_2$ is irrelevant (it might also be empty for $n=3$). Hence, if we were to add some points constrained to it (when $n>3$), the configuration essentially would not change and everything we said in this example would remain true.
\end{example}

\begin{definition}
    Let $Z = Z(\set{p_I},\set{\tilde u_I},\set{\tilde v_I})$ be a coordinate configuration in $\pe^{m,n}$ with coordinate subvarieties $L_1,\dots,L_k$, $k\geq1$. If $L=L_i$ has non-empty intersections with all other $L_j$, then we shall mean by the \emph{restriction} of $Z$ to $L$ the coordinate configuration $\restr ZL$ in $\pe^{m',n'}\simeq L$ that has $k-1$ subvarieties (by an abuse of notation, we index them with $\set{1,\dots,k}\setminus\set i$ instead of $\set{1,\dots,k-1}$) and is specified by $\set{p_{I\cup\set{i}}}$, $\set{\tilde u_I}$, $\set{\tilde v_I}$, $I\subseteq \set{1,\dots,k}\setminus\set i$.
\end{definition}
Geometrically, passing from $Z$ to $\restr ZL$ means we restrict our ambient space to just $L$, replace all other subvarieties with their intersections with $L$ and throw away all points that were not already constrained to a subset of $L_i$.

We will frequently be considering a \emph{family of coordinate configurations} (or just \emph{family} for short). By a family $A$, we mean a collection of coordinate configurations $A(m,n)$ in $\pe^{m,n}$, each with the same number of coordinate subvarieties, for $(m,n)\in S$, where $S\subseteq \zet_{\geq0}^2$ is some set; we say \uv{$A$ is defined on} the elements of $S$. As with individual coordinate configurations, specifying a family of them, up to permutation of coordinates, amounts to choosing a $k\in\zet_{\geq0}$ and providing a collection of non-negative integer-valued functions to describe configurations $A(m,n)$: the function $p_I:S\to \zet_{\geq0}$ for the numbers of double points constrained to each intersection, and then $\tilde u_I, \tilde v_I: S\to\zet_{\geq0}$ to describe the coordinate subvarieties.

\begin{definition}
    Let families $A$, $B$, such that $A$ has $k$ coordinate subvarieties and $B$ has $k+1$ coordinate subvarieties, be given, with $B$ defined on elements of $S_B$, and functions $M,N:S_B\to \zet_{\geq0}$. For an $(m,n)\in S_B$, we will say that $B$ is an \emph{$(M,N)$-inductant of $A$ at $(m,n)$}, if both $A(m,n)$ and $A(M(m,n),N(m,n))$ are defined and $B(m,n)$ arises from $A(m,n)$ by adding an extra coordinate subvariety $L=L_{k+1} \simeq \pe^{M(m,n),N(m,n)}\subset\pe^{m,n}$, specializing some of the double points from their original $\bigcap_{t\in I}L_t$ to $L\cap \bigcap_{t\in I}L_t$ in such a way that the restriction $\restr{B(m,n)}L$ would become $A(M(m,n),N(m,n))$, and finally removing any irrelevant points.

    We say that $B$ is an \emph{$(M,N)$-inductant of $A$}, if it is an $(M,N)$-inductant of $A$ at $(m,n)$ for any $(m,n)$ where all three $B(m,n)$, $A(m,n)$ and $A(M(m,n),N(m,n))$ are defined.
\end{definition}

While we allow $M$, $N$ to be arbitrary functions in the above definition, in our usage in the following sections, they will always be polynomials or quasipolynomials\footnote{We say a function is a \emph{$d$-quasipolynomial}, if it is a polynomial when restricted to any particular tuple of residue classes modulo $d$ of the arguments. A~function is a \emph{quasipolynomial} if it is a $d$-quasipolynomial for some $d$.} in $m$, $n$.

Let us describe how to interpret the inductant construction in terms of the functions $p_I$ and $\tilde u_I$, $\tilde v_I$ resp. $u_I$, $v_I$. Let $L_1,\dots,L_k$ be the coordinate subvarieties of $A(m,n)$, let $L_{k+1}$ be the new extra subvariety to be added. Let $p_I$ and $\tilde u_I$, $\tilde v_I$, resp. $u_I$, $v_I$ be the functions describing $A(m,n)$, whereas $p'_I$ and $\tilde u'_I$, $\tilde v'_I$, resp. $u'_I$, $v'_I$ will be the new functions to describe $B(m,n)$. Let us also write $M=M(m,n)$, $N=N(m,n)$ for short.

In order for each $L_t\cap L_{k+1} \subset L_{k+1}\simeq\pe^{M,N}$ to have the correct number of equations, we need the parameters for $I\not\ni k+1$ to be
\begin{equation}
    \label{eq:inductant-construction}
    \tilde u'_I(m,n) = \tilde u_I(M,N)
    \qquad\text{and}\qquad
    \tilde v'_I(m,n) = \tilde v_I(M,N),
\end{equation}
which leaves the remaining index sets, which may be written as $I\cup\set{k+1}$, with
\begin{equation}
    \tag{\ref{eq:inductant-construction}}
    \tilde u'_{I\cup\set{k+1}}(m,n) = \tilde u_I(m,n) - \tilde u_I(M,N)
    \qquad\text{and}\qquad
    \tilde v'_{I\cup\set{k+1}}(m,n) = \tilde v_I(m,n) - \tilde v_I(M,N).
\end{equation}
Translating to $u_I$ and $v_I$, we also obtain the relations
\begin{equation}
    \tag{\ref{eq:inductant-construction}}
    \begin{aligned}
        u'_I(m,n) &= u_I(m,n),\qquad & u'_{I\cup\set{k+1}}(m,n) &= u_I(m,n) - u_I(M,N),\\
        v'_I(m,n) &= v_I(m,n), & v'_{I\cup\set{k+1}}(m,n) &= v_I(m,n) - v_I(M,N)
    \end{aligned}
\end{equation}
for $I\subseteq\set{1,\dots,k}$. For the points, we need $p_I(M,N)$ from every $p_I(m,n)$ to be specialized to $L_{k+1}$. We must account for the possibility that these become irrelevant though, leading to relations
\begin{equation}
    \tag{\ref{eq:inductant-construction}}
    \begin{split}
    p'_I(m,n) &= p_I(m,n)-p_I(M,N),\\
    p'_{I\cup\set{k+1}}(m,n) &= \begin{cases}
        0, & \text{if $u'_{I\cup\set{k+1}}(m,n)+v'_{I\cup\set{k+1}}(m,n)=0$},\\
        p_I(M,N), & \text{otherwise}
    \end{cases}
    \end{split}
\end{equation}
for $I\subseteq\set{1,\dots,k}$.

\begin{example}
    \label{ex:inductant}
    Let us return to the configurations from Examples~\ref{ex:basic} and \ref{ex:defect} and view them as a family $A(m,n)$ defined on $(m,n)$ with $m\geq1$, $n\geq3$. As an example, let us construct its $(m-1,n)$-inductant family $B(m,n)$, then relations~\eqref{eq:inductant-construction} guide us to
    \begin{align*}
        \tilde u'_\emptyset&=m, & \tilde u'_{\set{3}} &= 1,\\
        \tilde u'_{I} &= 0\quad\text{for all other $I$},\span\omit\\
        \tilde v'_\emptyset&= n-3, & \tilde v'_{\set1} &= \tilde v'_{\set2} = 2,\\
        \tilde v'_{I} &= 0\quad\text{for all other $I$}.\span\omit
    \end{align*}
    In other words, we have added a new subvariety $L_3$ that corresponds to a hyperplane in the $\pe^m$ factor (therefore, it is general with respect to $L_1$, $L_2$ which corresponded to two-codimensional subspaces in the $\pe^n$ factor.)
    With points, we may initially set
    \begin{align*}
        p'_{\set1} &= p'_{\set2} = 1, & p'_{\set{1,3}} &= p'_{\set{2,3}} = m,\\
        p'_I &= 0\quad\text{for all other $I$},\span\omit
    \end{align*}
    but the intersections $L_1\cap L_3$ and $L_2\cap L_3$ are irrelevant, so we erase those points to end up with $p'_{\set{1,3}} = p'_{\set{2,3}} = 0$. We may then conclude that such a $B(m,n)$ makes sense for any $m\geq1$, $n\geq3$, but since $A(m,n)$ is defined on $m\geq1$, $n\geq3$, we conclude that $B$ is an $(m-1,n)$-inductant of $A$ only on $(m,n)$ with $m\geq2$, $n\geq3$.

    In Figure~\ref{fig:venn-inductant}, we visualize the inductant construction again using Venn diagrams. One may interpret the construction as splitting each cell into two smaller ones that sum to the original.
    \begin{figure}[tb]
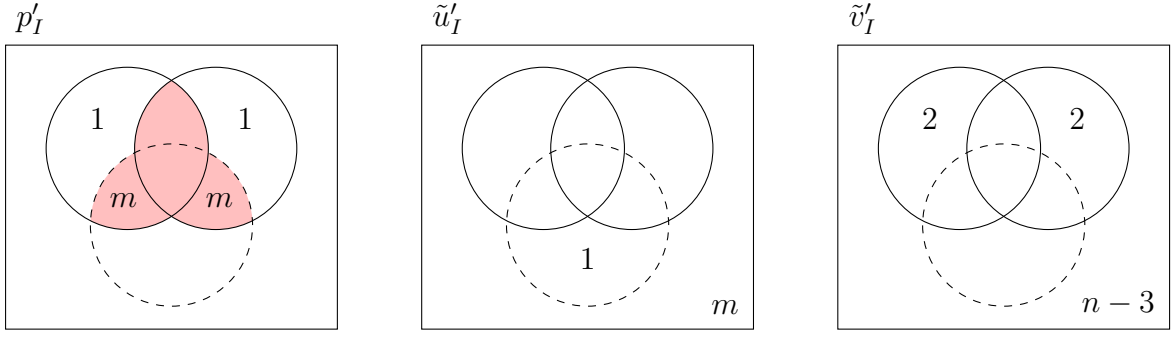

        \centering
        \threevenn{p'_I}{
            \fill[\tmpcolor](-.6,.1)++({\tmpangle}:1.1cm) arc ({180-\tmpangle}:{180+\tmpangle}:1.1cm) arc ({360-\tmpangle}:{360+\tmpangle}:1.1cm);
            \fill[\tmpcolor](-.6,.1)++({\tmpangle-60}:1.1cm) arc ({120-\tmpangle}:{120+\tmpangle}:1.1cm) arc ({300-\tmpangle}:{300+\tmpangle}:1.1cm);
            \fill[\tmpcolor](.6,.1)++({240-\tmpangle}:1.1cm) arc ({240-\tmpangle}:{240+\tmpangle}:1.1cm) arc ({60-\tmpangle}:{60+\tmpangle}:1.1cm);
            \path
                (-1,.5) node{$1$}
                (1,.5) node{$1$}
                (-.65,-.6) node{$m$}
                (.65,-.6) node{$m$}
            ;
        }
        \qquad
        \threevenn{\tilde u'_I}{
            \path
                (2.2,-2.3) node[above left]{$m$}
                (0,-1.4) node{$1$}
            ;
        }
        \qquad
        \threevenn{\tilde v'_I}{
            \path
                (2.2,-2.3) node[above left]{$n-3$}
                (-1,.5) node{$2$}
                (1,.5) node{$2$}
            ;
        }
        \caption{The parameter functions specifying the family $B(m,n)$ in Example~\ref{ex:inductant}, visualized in Venn diagrams. The newly added subvariety $L_3$ is represented by a dashed circle and we leave cells blank instead of writing zeros. Further, in $p'_I$ we list the \uv{initial} values before erasure of irrelevant points and mark the cells where this erasure takes place with a red background.}
        \label{fig:venn-inductant}
    \end{figure}
\end{example}

From relations \eqref{eq:inductant-construction}, we may see that a poor choice of $(M,N)$ may lead to $\tilde u'_I$, $\tilde v'_I$ or $p'_I$ in these relations having values incompatible with the definition of a coordinate configuration for many $(m,n)$, thus allowing us to define the inductant only on a small set $S_B\subseteq\zet_{\geq0}$, or perhaps even nowhere at all.
As long as the inductant $B$ is defined however, it allows us to perform an inductive proof of non-defectivity of $A$, which is the motivation for its name. This is essentially contained in the following two lemmata:
\begin{lemma}
    \label{lem:defect-additivity}
    If $B$ is an $(M,N)$-inductant of a family $A$, then
    \[
        \vdim B(m,n) = \vdim A(m,n) - \vdim A(M(m,n),N(m,n))
    \]
    holds whenever $B$ is defined on $(m,n)$.
\end{lemma}
\begin{proof}
    Throughout this proof, let $M:=M(m,n)$, $N:=N(m,n)$ for short.
    Let us deal with the parts of virtual dimensions coming from points and the parts coming from the ideals of subvarieties separately. Let $L_1,\dots,L_k$ be the coordinate subvarieties of $A(m,n)$ and let $L_{k+1}$ the additional subvariety present in $B(m,n)$. Then we will view $A(M,N)$ in $L_{k+1}\simeq \pe^{M,N}$, in this way it has the coordinate subvarieties $L'_t = L_t\cap L_{k+1}$ for $t=1,\dots,k$. Abbreviating
    \[
        \mathcal I_{L_B} := \mathcal I_{\bigcup_{t\in \set{1,\dots,k+1}} L_t}, \qquad \mathcal I_{L_A} := \mathcal I_{\bigcup_{t\in \set{1,\dots,k}} L_t}, \qquad \mathcal I_{L_{A'}} := \mathcal I_{\bigcup_{t\in \set{1,\dots,k}} L'_t},
    \]
    we can then observe the exact sequence
    \begin{equation}
    \label{eq:subvariety-ideals-exactseq}
    \begin{tikzcd}
        0 \arrow[r] & \mathcal I_{L_B} \arrow[r] & \mathcal I_{L_A} \arrow[r] & \mathcal I_{L_{A'}} \arrow[r] & 0,
    \end{tikzcd}
    \end{equation}
    where $\mathcal I_{L_B} \to \mathcal I_{L_A}$ is just the inclusion and $\mathcal I_{L_A} \to \mathcal I_{L_{A'}}$ is the restriction $f\mapsto \restr f{L_{k+1}}$. Further, these maps preserve bidegree, so we may take the exact sequence just in the bidegree $(1,2)$ homogeneous parts and obtain
    \begin{equation}
        \label{eq:inductant-subvar-ideals}
        \dim \mathcal I_{L_B}(1,2) = \dim\mathcal I_{L_A}(1,2) - \dim\mathcal I_{L_{A'}}(1,2)
    \end{equation}
    from the exact sequence.

    Now let us consider any double point $P$ of $A(m,n)$ constrained to some $\bigcap_{t\in I}L_t$. In the construction of $B(m,n)$, it might either be left untouched or it might be further constrained to $L_{k+1}$. In the former case, its contribution to $\vdim B(m,n)$ is the same as its contribution to $\vdim A(m,n)$, while it is not present in $\vdim A(M,N)$ at all. In the latter case, its (negative) contributions to $\vdim B(m,n)$, $\vdim A(m,n)$, $\vdim A(M,N)$ are (in the notation of \eqref{eq:inductant-construction})
    \begin{gather*}
        \min\set{u'_{I\cup\set{k+1}}(m,n)+v'_{I\cup\set{k+1}}(m,n), m+n+1},
        \\
        \min\set{u_I(m,n)+v_I(m,n), m+n+1},
        \\
        \min\set{u_I(M,N)+v_I(M,N), M+N+1}
    \end{gather*}
    respectively. Implicitly, based on the fact that we are in fact handling some point, \eqref{eq:minimum-simplifies} takes effect, so we know that $m+n+1$ (resp. $M+N+1$) is achieved in the above minima only if the respective index set is empty.
    Note that $I\cup\set{k+1}$ cannot be empty, so the first expression is always just
    \[
        u'_{I\cup\set{k+1}}(m,n)+v'_{I\cup\set{k+1}}(m,n) = u_I(m,n)+v_I(m,n) - u_I(M,N) - v_I(M,N).
    \]
    Now it becomes clear that for $I\neq\emptyset$, this will equal the difference
    \[
        \zav{u_I(m,n)+v_I(m,n)} - \zav{u_I(M,N)+v_I(M,N)},
    \]
    which is the difference of the other two expressions; whereas for $I=\emptyset$, we will get
    \begin{align*}
        (m+n+1)-(M+N+1) &= (m+1)+(n+1)-(M+1)-(N+1) =\\&= u_{\emptyset}(m,n)+v_{\emptyset}(m,n) - u_{\emptyset}(M,N) - v_{\emptyset}(M,N).
    \end{align*}
    So in either case, we obtain that the contribution of $P$ to $\vdim B(m,n)$ equals the difference of its contributions to $\vdim A(m,n)$ and $\vdim A(M,N)$.

    Putting the results on ideals and on contributions of points together, we get $\vdim B(m,n) = \vdim A(m,n) - \vdim A(M,N)$.
\end{proof}

\begin{lemma}
    \label{lem:inductant-induction}
    Let $B$ be an $(M,N)$-inductant of $A$ and suppose $B$ is defined on $(m,n)$. If $B(m,n)$ and $A(M(m,n),N(m,n))$ are both non-defective and subabundant (resp. superabundant, resp. equiabundant), then so is $A(m,n)$.
\end{lemma}
\begin{proof}
    Let us again write $M:=M(m,n)$, $N:=N(m,n)$ for short.
    The fact that $A(m,n)$ will have the same abundancy as $B(m,n)$ and $A(M,N)$ follows immediately from Lemma~\ref{lem:defect-additivity}, so let us focus on showing $A(m,n)$ will also be non-defective.

    Let us prove non-defectivity of $A(m,n)$ in the case when the configurations are superabundant; the subabundant version follows analogously, and the equiabundant version follows from their conjunction.
    Observe that akin to the short exact sequence \eqref{eq:subvariety-ideals-exactseq}, we have the exact sequence
    \begin{equation}
    \label{eq:ideals-exactseq}
    \begin{tikzcd}
        0 \arrow[r] & \mathcal I_{B(m,n)}(1,2) \arrow[r] & \mathcal I_{A(m,n)}(1,2) \arrow[r] & \mathcal I_{A(M,N)}(1,2),
    \end{tikzcd}
    \end{equation}
    where the map $\mathcal I_{B(m,n)}(1,2) \to \mathcal I_{A(m,n)}(1,2)$ is the inclusion map and $\mathcal I_{A(m,n)}(1,2) \to \mathcal I_{A(M,N)}(1,2)$ is the restriction to the extra subvariety used in the inductant construction, which is then identified with $\pe^{M,N}$. Note that with the introduction of conditions given by the double points, this last map need not be surjective, but we may still extract an inequality of dimensions
    \[
        \dim\mathcal I_{A(m,n)}(1,2) \leq \dim\mathcal I_{B(m,n)}(1,2)+\dim\mathcal I_{A(M,N)}(1,2).
    \]
    In the superabundant case, non-defectivity of $B(m,n)$ and $A(M,N)$ means
    \[
        \dim\mathcal I_{B(m,n)}(1,2) = \dim\mathcal I_{A(M,N)}(1,2) = 0,
    \]
    so the inequality gives $\dim\mathcal I_{A(m,n)}(1,2)\leq 0$, i.e. the ideal vanishes in the bidegree $(1,2)$ part. Since we already concluded $A(m,n)$ must be superabundant, this means it is non-defective.
\end{proof}

\begin{remark}[commuting inductants]
\label{rmrk:commuting-inductants}
Suppose $(M_1,N_1)$, $(M_2,N_2)$ are two pairs of functions that commute at some $(m,n)\in\zet_{\geq0}^2$, i.e.
\begin{multline*}
    (M_1(M_2(m,n),N_2(m,n)), N_1(M_2(m,n),N_2(m,n))) =\\= (M_2(M_1(m,n),N_1(m,n)), N_2(M_1(m,n),N_1(m,n))).
\end{multline*}
Let us argue that the two corresponding inductant constructions then also commute at $(m,n)$, in the sense that if $A$ is a family of coordinate configurations that is defined on all four
\begin{gather*}
(m,n),\quad (m_1,n_1):= (M_1(m,n),N_1(m,n)), \quad (m_2,n_2):=(M_2(m,n),N_2(m,n)),\\ (m_3,n_3):= (M_1(m_2,n_2),N_1(m_2,n_2)) = (M_2(m_1,n_1),N_2(m_1,n_1)),
\end{gather*}
$B_1$ is an $(M_1,N_1)$-inductant of $A$ at $(m,n)$ and $(m_2,n_2)$, $B_2$ is an $(M_2,N_2)$-inductant of $A$ at $(m,n)$ and $(m_1,n_1)$ and $C$ is an $(M_2,N_2)$-inductant of $B_1$ at $(m,n)$, then $C$ is also an $(M_1,N_1)$-inductant of $B_2$ at $(m,n)$ -- see also Figure~\ref{fig:commuting-inductants}.

\begin{figure}[tb]
    \begin{tabular}{*{2}{m{\containmentwd}}}
        \vtop{\hsize=\containmentwd\hbox to\hsize{\hss
            \begin{tikzcd}[sep=2cm]
                (m_3,n_3)\arrow[r, mapsfrom, "\tinylabel{$\textstyle(M_1,N_1)$}"]\arrow[d, mapsfrom, "\tinylabel{$\textstyle(M_2,N_2)$}"] & (m_2,n_2)\arrow[d, mapsfrom, "\tinylabel{$\textstyle(M_2,N_2)$}"]\\
                (m_1,n_1)\arrow[r, mapsfrom, "\tinylabel{$\textstyle(M_1,N_1)$}"] & (m,n)
            \end{tikzcd}
        \hss}}
        &
        \vtop{\hsize=\containmentwd\hbox to\hsize{\hss
            \begin{tikzpicture}[n/.style = {draw, circle, inner sep=2pt}, a/.style = {-{Latex[scale=1]}, shorten <=2pt, shorten >=2pt}, scale=2.35]
                \node[n] (A) at (1,0) {\familynameCommuting A};
                \node[n] (B1) at (0,0) {\familynameCommuting{B_1}};
                \node[n] (B2) at (1,1) {\familynameCommuting{B_2}};
                \node[n] (C) at (0,1) {\familynameCommuting C};

                \draw[a](B1) -- (A) node[midway, above=.05cm] {\tinylabel{$(M_1,N_1)$}};
                \draw[a](B2) -- (A) node[midway, right=.05cm] {\tinylabel{$(M_2,N_2)$}};
                \draw[a](C) -- (B1) node[midway, right=.05cm] {\tinylabel{$(M_2,N_2)$}};
                \draw[a,dashed](C) -- (B2) node[midway, above=.05cm] {\tinylabel{$(M_1,N_1)$}};
            \end{tikzpicture}
        \hss}}
    \end{tabular}
    \caption{On the left, $(M_1,N_1)$ and $(M_2,N_2)$ commute in $(m,n)$ as functions $\zet^2\to\zet^2$. On the right, arrows indicate inductant constructions, with the inductant family pointing to the family is an inductant of. The three full arrows determine $B_1$, $B_2$, and $C$, we then claim that the inductant relation indicated by the dashed arrow also holds.}
    \label{fig:commuting-inductants}
\end{figure}
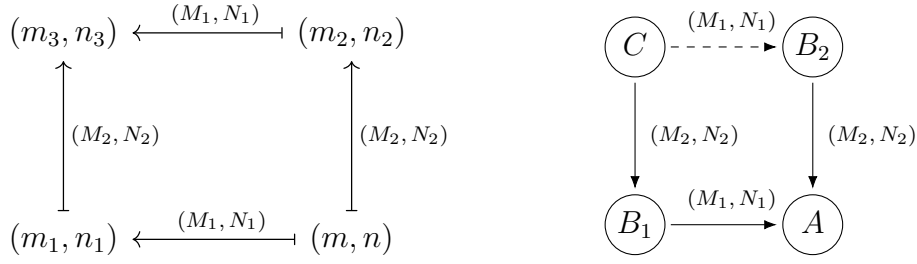

This is essentially because when evaluating the parameters in the inductant families through \eqref{eq:inductant-construction}, only the values at $(m,n)$, $(m_1,n_1)$, $(m_2,n_2)$, $(m_3,n_3)$ matter: Let $A$ have $k$ subvarieties and let $\tilde u_I^{(A)}$, $\tilde v_I^{(A)}$, $p_I^{(A)}$ be its parameter functions; similarly, let us denote the parameter functions for the other families with an analogous superscript. For a more clear notation, let us also index the subvarieties in $B_2$ with $\set{1,\dots,k,k+2}$ instead of $\set{1,\dots,k,k+1}$. If we interpret the subvariety $L_{k+1}$ in $C$ as the last one added instead of $L_{k+2}$, then $C$ indeed becomes an $(M_1,N_1)$-inductant of $B_2$, since (denoting the family to which a function corresponds with a superscript) for example
\begin{align*}
    \tilde u^{(C)}_{I}(m,n) &= \tilde u^{(B_1)}_I(m_2,n_2) = \tilde u^{(A)}_I(m_3,n_3) = \tilde u^{(B_2)}_I(m_1,n_1),\\
    \tilde u^{(C)}_{I\cup\set{k+1}}(m,n) &= \tilde u^{(B_1)}_{I\cup\set{k+1}}(m_2,n_2) = \tilde u^{(A)}_I(m_2,n_2) - \tilde u^{(A)}_I(m_3,n_3) =\\
    &= \tilde u^{(B_2)}_I(m,n) - \tilde u^{(B_2)}_I(m_1,n_1),\\
    \tilde u^{(C)}_{I\cup\set{k+2}}(m,n) &= \tilde u^{(B_1)}_{I}(m,n) - \tilde u^{(B_1)}_{I}(m_2,n_2) = \tilde u^{(A)}_{I}(m_1,n_1) - \tilde u^{(A)}_{I}(m_3,n_3) =\\&= \tilde u^{(B_2)}_{I\cup\set{k+2}}(m_1,n_1),\\
    \tilde u^{(C)}_{I\cup\set{k+1,k+2}}(m,n) &= \tilde u^{(B_1)}_{I\cup\set{k+1}}(m,n) - \tilde u^{(B_1)}_{I\cup\set{k+1}}(m_2,n_2) =\\
    &= \tilde u^{(A)}_I(m,n) - \tilde u^{(A)}_I(m_1,n_1) - \tilde u^{(A)}_I(m_2,n_2) + \tilde u^{(A)}_I(m_3,n_3)=\\
    &= \tilde u^{(A)}_I(m,n) - \tilde u^{(A)}_I(m_2,n_2) - \tilde u^{(A)}_I(m_1,n_1) + \tilde u^{(A)}_I(m_3,n_3)=\\
    &= \tilde u^{(B_2)}_{I\cup\set{k+1}}(m,n) - \tilde u^{(B_2)}_{I\cup\set{k+1}}(m_1,n_1)
\end{align*}
for all $I\subseteq\set{1,\dots,k}$. Similar relations arise for $\tilde v_I$ and $p_I$, justifying that $C$ would also come about as an $(M_1,N_1)$-inductant from $B_2$.
\end{remark}

\begin{remark}
The definitions we presented in this section are similar to the concept of Brambilla-Ottaviani lattices as defined by Torrance and Vannieuwenhoven in \cite{torrance-vannieuwenhoven2021}, which is named after and also aims to generalize the inductive proof of the cubic case of Alexander-Hirschowitz theorem due to Brambilla and Ottaviani \cite{brambilla-ottaviani}. In contrast to the language of Brambilla-Ottaviani lattices, viewing our situation through (families of) coordinate configurations and inductants will allow us to mix different kinds of inductant \uv{steps}, which will prove useful in the following section.
\end{remark}

\section{Superabundant induction strategy for $\mathcal O(1,2)$}
\label{sec:strategies}

\hypertarget{def:A0}{Our main focus in this paper is the following family: Let $A_0(m,n)$ be defined on $(m,n)$ with $n\geq m-2$ and let it have no subvarieties and $\ups(m,n)=\floor{\frac{(m+1)(n-m+2)}2}+1$ double points chosen generally in the whole $\pe^{m,n}$.}
Splitting the behavior of $\ups(m,n)$ into nice and ugly cases, we can express the virtual dimension of $A_0(m,n)$ as
\begin{equation}
\label{eq:A0defect}
\vdim A_0(m,n) = \begin{cases}
    3\binom{m+1}3-(n+m+1), & \text{if $(m,n)$ is nice,}\\
    3\binom{m+1}3-\frac{n+m+1}2, & \text{if $(m,n)$ is ugly.}
\end{cases}
\end{equation}
In particular, $A_0(m,n)$ is superabundant for $n\geq 3\binom{m+1}3-m-1$ for nice $(m,n)$ and for $n\geq 6\binom{m+1}3-m-1$ for ugly $(m,n)$. In both cases, the bound is a cubic polynomial in $m$. Note that we know that $\ups (m,n) = \upr(m,n)$ under these same bounds (by Lemma~\ref{lem:arithmetic}) and recall that non-defectivity of $A_0(m,n)$ just means that $\sigma_{\ups(m,n)}(X_{m,n})$ fills the ambient space. In the superabundant case, this immediately implies non-defectivity of $\sigma_s(X_{m,n})$ for all $s\geq \ups(m,n)$, so once $n$ satisfies the above bounds and we have $\ups(m,n)=\upr(m,n)$, then non-defectivity of $A_0(m,n)$ immediately implies that all superabundant secant varieties of $X_{m,n}$ are non-defective.

We will try to prove non-defectivity of $A_0(m,n)$ for all $(m,n)$ with $m\geq 3$ and $n$ bounded below by a certain cubic polynomial in $m$. Our strategy will be to prove non-defectivity of $A_0(m,n)$ by an induction using two inductant families constructed from $A_0$. The proof of non-defectivity for these inductant families will require inductants of their own, etc. We will terminate this process at families where $\mathcal I_{\bigcup_{t\in \set{1,\dots,k}}L_t}(1,2)$ vanishes and which have no double points, these points having become irrelevant in the inductant construction.

A~partial inductive strategy to prove non-defectivity of $A_0$ is in essence already present in \cite{abo-brambilla}. Though the methods therein differ from ours, we believe an alternative proof could be given using the technique of inductants. We comment on this in Remark~\ref{rmrk:vertical-alternative}, but for now we refer to \cite{abo-brambilla} as much as possible so as not to encumber our presentation unnecessarily.

We define the following family of configurations: let $A_1(m,n)$ for $n\geq m$ have a single subvariety $L$ given by $y_0=y_1=0$ and let it have $m+1$ double points supported on the whole $\P^{m,n}$ and $\ups(m,n-2)$ double points supported on $L$. This is an $(m,n-2)$-inductant of $A_0$. Comparing virtual dimensions, we then see that
\[
\vdim A_1(m,n) = \begin{cases}
    -2, & \text{if $(m,n)$ is nice},\\
    -1, & \text{if $(m,n)$ is ugly}.
\end{cases}
\]
In particular, it is always superabundant.

\begin{prop}
    \label{prop:A1}
    $A_1(m,n)$ is non-defective for all $n\geq m\geq 3$. In particular, if $n\geq m\geq3$ and $A_0(m,n-2)$ is superabundant and non-defective, then so is $A_0(m,n)$.
\end{prop}
Let us note that we will actually use this proposition only for nice pairs $(m,n)$. However, we think that including the ugly case as well makes for a more concise statement and does not make the proof much more difficult.
\begin{proof}
    Consider the family $A_1'$ that only differs from $A_1$ by having one less point on $L$, this will be either equiabundant or subabundant due to
    \[
    \vdim A'_1(m,n) = \begin{cases}
        0, & \text{if $(m,n)$ is nice},\\
        1, & \text{if $(m,n)$ is ugly}.
    \end{cases}
    \]
    Reformulating \cite[Corollary 3.12]{abo-brambilla} in the language of inductants, $A_1'(m,n)$ is non-defective for $n\geq m\geq1$, hence $\dim\mathcal I_{A'_1(m,n)}(1,2)=0$ resp. $1$ in these cases. For the former case, there is nothing to prove, since adding an extra double point leaves $\mathcal I_{A_1(m,n)}(1,2)=0$.

    In the latter case, the vector space $\mathcal I_{A'_1(m,n)}(1,2)$ is generated by a single polynomial, say $f\neq0$. Let $P\in L$ be the extra double point that $A_1(m,n)$ has compared to $A'_1(m,n)$, let $P_0,\dots,P_m\in \pe^{m,n}$ be the unconstrained points of $A'_1(m,n)$ and suppose for the sake of contradiction that $\mathcal I_{A_1(m,n)}(1,2)\neq0$. This means that we already had $\frac{\partial f}{\partial y_0}(P)=\frac{\partial f}{\partial y_1}(P)=0$, and since this holds for a general choice of $P\in L$, we have $\frac{\partial f}{\partial y_0},\frac{\partial f}{\partial y_1}\in (y_0,y_1)$. Together with $f\in (y_0,y_1)$, this means $f\in U\otimes \gener{y_0^2,y_0y_1,y_1^2}$.
    To reach a contradiction, let us show that this is not the case for a general choice of points in the configuration $A'_1(m,n)$.

    Suppose that $L=\pe(U)\times\pe(W)$ for a two-codimensional coordinate subspace $W\leq V$. Then we may view $U\otimes \gener{y_0^2,y_0y_1,y_1^2}$ as the bidegree $(1,2)$ homogeneous part of $\Sym U\otimes\Sym(W^\perp)$, which is the bigraded algebra associated to $\pe(U)\times\pe(W^\perp)\simeq \pe^{m,1}$. Then $f$ is a non-zero polynomial in this algebra and it is singular at the $m+1$ projections of $P_0,\dots,P_m$ onto $\pe(U)\times\pe(W^\perp)$. This would mean the $(m+1)$-th secant variety of $X_{m,1}$ does not fill its ambient space, which contradicts \cite[Theorem 3.3 (4)]{catalisano-geramita-gimigliano} for $m\geq3$. Therefore, $f$ could not have lied in $U\otimes \gener{y_0^2,y_0y_1,y_1^2}$, and thus $\mathcal I_{A_1(m,n)}(1,2)=0$.
\end{proof}

With this result in hand, we mainly need to provide an induction in which $m$ increases.
Because of the arithmetic considerations we discussed in Section~\ref{sec:preliminaries}, we treat the nice and ugly pairs $(m,n)$ separately, providing different strategies utilizing different systems of inductant families to facilitate the proof.

\subsection{The nice case}
\label{subsec:nice}

In our strategy for the nice cases, we will use several families of coordinate configurations that arise from $A_0$ through a series of inductant constructions. These relationships are captured in Figure~\ref{fig:nicestrategy}. Note that while the aim of the strategy is to handle $A_0(m,n)$ in the nice cases, it will be advantageous to use some of the other configurations even in the ugly cases. This is essentially because the nice--ugly distinction looses significance as some points become irrelevant and are erased during the successive inductant constructions.

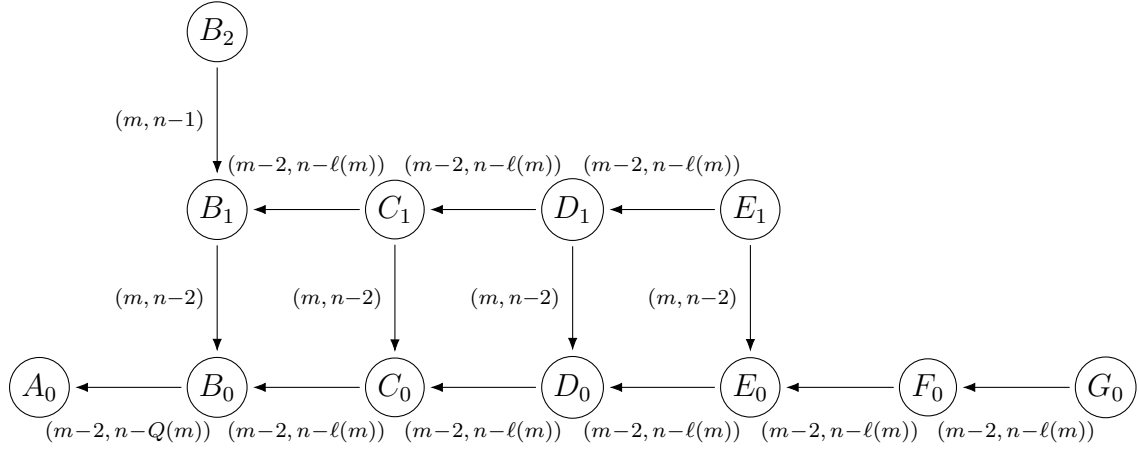
\begin{figure}[tb]
    \centering
    \begin{tikzpicture}[n/.style = {draw, circle, inner sep=2pt}, a/.style = {-{Latex[scale=1]}, shorten <=2pt, shorten >=2pt}, scale=2.35]
        \node[n] (A0) at (0,0) {\familyname A0};
        \node[n] (B0) at (1,0) {\familyname B0};
        \node[n] (C0) at (2,0) {\familyname C0};
        \node[n] (D0) at (3,0) {\familyname D0};
        \node[n] (E0) at (4,0) {\familyname E0};
        \node[n] (F0) at (5,0) {\familyname F0};
        \node[n] (G0) at (6,0) {\familyname G0};
        \node[n] (B1) at (1,1) {\familyname B1};
        \node[n] (B2) at (1,2) {\familyname B2};
        \node[n] (C1) at (2,1) {\familyname C1};
        \node[n] (D1) at (3,1) {\familyname D1};
        \node[n] (E1) at (4,1) {\familyname E1};

        \draw[a](B0) -- (A0) node[midway, below=.3cm] {\tinylabel{$(m-2,n-Q(m))$}};
        \draw[a](C0) -- (B0) node[midway, below=.3cm] {\tinylabel{$(m-2,n-\ell(m))$}};
        \draw[a](D0) -- (C0) node[midway, below=.3cm] {\tinylabel{$(m-2,n-\ell(m))$}};
        \draw[a](E0) -- (D0) node[midway, below=.3cm] {\tinylabel{$(m-2,n-\ell(m))$}};
        \draw[a](F0) -- (E0) node[midway, below=.3cm] {\tinylabel{$(m-2,n-\ell(m))$}};
        \draw[a](G0) -- (F0) node[midway, below=.3cm] {\tinylabel{$(m-2,n-\ell(m))$}};

        \draw[a](C1) -- (B1) node[midway, above=.3cm] {\tinylabel{$(m-2,n-\ell(m))$}};
        \draw[a](D1) -- (C1) node[midway, above=.3cm] {\tinylabel{$(m-2,n-\ell(m))$}};
        \draw[a](E1) -- (D1) node[midway, above=.3cm] {\tinylabel{$(m-2,n-\ell(m))$}};

        \draw[a] (B2)--(B1) node[midway, left=.1cm] {\tinylabel{$(m,n-1)$}};
        \draw[a] (B1)--(B0) node[midway, left=.1cm] {\tinylabel{$(m,n-2)$}};
        \draw[a] (C1)--(C0) node[midway, left=.1cm] {\tinylabel{$(m,n-2)$}};
        \draw[a] (D1)--(D0) node[midway, left=.1cm] {\tinylabel{$(m,n-2)$}};
        \draw[a] (E1)--(E0) node[midway, left=.1cm] {\tinylabel{$(m,n-2)$}};
    \end{tikzpicture}
    \caption{Families of coordinate configurations in the inductive strategy for the nice case. An arrow labeled with $(M,N)$ denotes that a family is an $(M,N)$-inductant of the family being pointed to. We denote $Q(m)=3m^2-6m+2$, $\ell(m)=12m-22$ for short.}
    \label{fig:nicestrategy}
\end{figure}

Though these inductant constructions essentially determine these families, let us go through them one by one to specify precisely where we define them, along with some further details. Here, we choose to describe each family using the parameters $\tilde u_I$, $\tilde v_I$, since from them, it is more readily seen that the configurations make sense for the given $(m,n)$.
Two particular univariate polynomials will appear prominently, we denote them $Q(m)=3m^2-6m+2$ and $\ell(m)=12m-22$ for short.

\begin{itemize}
\famitem B0{
For any nice $(m,n)$ with $m\geq 2$ and $n\geq\max\set{3m^2-5m-2, Q(m)}$ (note that this equals $3m^2-5m-2$ for $m\geq4$), we define $B_0(m,n)$ with $k=1$ coordinate subvariety through
\begin{align*}
    \tilde u_\emptyset &= m-1, & \tilde u_{\set1} &= 2, \\
    \tilde v_\emptyset &= n+1-Q(m), & \tilde v_{\set1} &= Q(m),\\
    p_\emptyset &= \ups(m,n) - \ups(m-2,n-Q(m)) = & p_{\set1} &= \ups(m-2, n-Q(m)) = \\
    &=  n + \frac{3m^3-9m^2+4m+4}{2},
    &&= \frac{mn-n-3m^3+8m^2-3m}{2}.
\end{align*}
Comparing with the relations \eqref{eq:inductant-construction}, it is straightforward to verify that $B_0$ is an $(m-2, n-Q(m))$-inductant of $A_0$. By using Lemma~\ref{lem:defect-additivity} with \eqref{eq:A0defect}, we obtain
\begin{align*}
    \vdim B_0(m,n) &= \vdim A_0(m,n) - \vdim A_0(m-2,n-Q(m)) =\\&= 3\zav{\binom{m+1}3-\binom{m-1}3} - (Q(m)+2) = -1,
\end{align*}
making all $B_0(m,n)$ superabundant.}
\par
Note that within this definition, the bound $n\geq Q(m)$ is given to ensure that $\tilde v_\emptyset > 0$, whereas $n\geq 3m^2-5m-2$ is to ensure $p_{\set1}\geq 0$. This stems from $\ups(m,n)\geq 1\iff n-m+2\geq0$, hence $\ups(m-2, n-Q(m))\geq 1\iff n\geq Q(m)+(m-2)-2 = 3m^2-5m-2$. Similar bounds chosen for the sake of non-negativity of the individual parameters will appear in some of the following families of configurations as well.

\famitem B1{
For any  $(m,n)$ with $m\geq2$, $n\geq Q(m)+1$, we define $B_1(m,n)$ with $k=2$ coordinate subvarieties through
\begin{align*}
    \tilde u_\emptyset &= m-1, & \tilde u_{\set1} &= 2,\\
    \tilde u_{\set2} &= 0, & \tilde u_{\set{1,2}} &= 0,\\
    \tilde v_\emptyset &= n-1-Q(m), & \tilde v_{\set{1}} &= Q(m),\\
    \tilde v_{\set2} &= 2, & \tilde v_{\set{1,2}} &= 0,\\
    p_\emptyset &= 2, & p_{\set1} &= m-1,\\
    p_{\set2} &= n+\frac{3m^3-9m^2+4m}{2}, & p_{\set{1,2}} &= 0.
\end{align*}
This is an $(m,n-2)$-inductant of $B_0$; note that $p_{\set{1,2}}=0$ is due to erasure of irrelevant points. Note further that we define $B_1(m,n)$ even for ugly $(m,n)$, with the parameter functions indeed giving integer values even when $(m,n)$ is ugly. Applying Lemma~\ref{lem:defect-additivity} at nice $(m,n)$, we obtain $\vdim B_1(m,n) = 0$ since all $B_0$ have virtual dimension $-1$. Since all parameter functions of $B_1$ are polynomials in $m$, $n$, the virtual dimension $\vdim B_1(m,n)$ must also be a polynomial, hence it must be $0$ even for ugly $(m,n)$. Thus $B_1$ is equiabundant.
}

\famitem B2{
For any $(m,n)$ with $m\geq2$, $n\geq Q(m)+2$, we define $B_2(m,n)$ with $k=3$ coordinate subvarieties through
\begin{align*}
    \tilde u_\emptyset &= m-1, & \tilde u_{\set1} &= 2,\\
    \tilde u_I &= 0 \quad\text{for all other $I$},\span\omit\\
    \tilde v_\emptyset &= n-2-Q(m), & \tilde v_{\set1} &= Q(m),\\
    \tilde v_{\set2} &= 2, & \tilde v_{\set3} &= 1,\\
    \tilde v_I &= 0 \quad\text{for all other $I$},\span\omit\\
    p_{\set3} &= 2, & p_{\set{2}} &= 1,\\
    p_I &= 0 \quad\text{for all other $I$}.\span\omit\\
\end{align*}
This is an $(m,n-1)$-inductant of $B_1$. Being an inductant of an equiabundant family, it is also equiabundant.
}

\famitem C0{
For any nice $(m,n)$ with $m\geq 4$, $n\geq\max\set{3m^2-5m-2, Q(m)+2}$ (note that this equals $3m^2-5m-2$ for $m\geq6$), we define $C_0(m,n)$ with $k=2$ coordinate subvarieties through
\begin{align*}
    \tilde u_\emptyset &= m-3, & \tilde u_{\set1} &= 2,\\
    \tilde u_{\set2} &= 2, & \tilde u_{\set{1,2}} &= 0,\\
    \tilde v_{\emptyset} &= n-1-Q(m), & \tilde v_{\set1} &= Q(m-2) = 3m^2-18m+26,\\
    \tilde v_{\set2} &= 2, & \tilde v_{\set{1,2}} &= 12m-24,\\
    p_\emptyset &= 9m^2-24m+12, & p_{\set1} &= n-3m^2+5m+2,\\
    p_{\set2} &= n + \frac{3m^3-27m^2+52m-20}{2}, & p_{\set{1,2}} &= \ups(m-4, n-2-Q(m)) =\\&&&= \frac{mn-3n-3m^3+14m^2-13m-4}2.
\end{align*}
This is an $(m-2,n-\ell(m))$-inductant of $B_0$ and hence equiabundant, since $B_0$ has constant virtual dimension. In fact, from this point on, all families we name in this subsection will be equiabundant for the same reasons, so we shall stop noting it.
}

\famitem C1{
For any $(m,n)$ with $m\geq4$, $n\geq Q(m)+3$, we define $C_1(m,n)$ with $k=3$ coordinate subvarieties through
\begin{align*}
    \tilde u_\emptyset &= m-3, & \tilde u_{\set1} &= \tilde u_{\set2} = 2,\\
    \tilde u_I &= 0\quad \text{for all other $I$},\span\omit\\
    \tilde v_\emptyset &= n-3-Q(m), & \tilde v_{\set1} &= Q(m-2),\\
    \tilde v_{\set2} &= \tilde v_{\set3} = 2, & \tilde v_{\set{1,2}} &= 12m-24,\\
    \tilde v_I &= 0\quad\text{for all other $I$},\span\omit\\
    p_{\set1} &= p_{\set2} = 2, & p_{\set{1,2}} &= m-3,\\
    p_{\set3} &= 9m^2-24m+12, & p_I &= 0\quad\text{for all other $I$}.
\end{align*}
This is both an $(m-2,n-\ell(m))$-inductant of $B_1$ (with $L_2$ being the newly added subvariety in the inductant construction) and, due to Remark~\ref{rmrk:commuting-inductants}, also an $(m,n-2)$-inductant of $C_0$ (with $L_3$ being the newly added subvariety). Note that $p_{\set{1,3}} = p_{\set{2,3}} = p_{\set{1,2,3}} = 0$ due to erasure of irrelevant points.
}

\famitem D0{
Since $D_0$, which we define on nice $(m,n)$ with $m\geq6$, $n\geq\max\set{3m^2-5m-2, Q(m)+4}$ (note that this equals $3m^2-5m-2$ for $m\geq8$) and which has $k=3$ coordinate subvarieties, comes about by performing a second inductant construction on $C_0$ with $(M,N)=(m-2,n-\ell(m))$, there is a symmetry between the parameters whose indices differ only by permuting these subvarieties corresponding to the same inductant step. We use this symmetry to condense the description, which would otherwise at this point be starting to become quite cumbersome -- in the following, $J$ is always a subset of $\set{2,3}$:
\begin{align*}
    \tilde u_{\emptyset} &= m-5, & \tilde u_{\set1} &= \tilde u_{\set2} = \tilde u_{\set3} = 2,\\
    \tilde u_I &= 0\quad\text{for all other $I$},\span\omit\\
    \tilde v_J &= \begin{cases}
        n-3-Q(m), & \text{if $\abs J=0$,}\\
        2, & \text{if $\abs J = 1$,}\\
        0, & \text{if $\abs J = 2$,}
    \end{cases}
    &
    \tilde v_{\set1\cup J} &= \begin{cases}
        Q(m-4), & \text{if $\abs J = 0$,}\\
        12m-48, & \text{if $\abs J = 1$,}\\
        24, & \text{if $\abs J = 2$,}
    \end{cases}\\
    p_J &= \begin{cases}
        36m-84, & \text{if $\abs J=0$,}\\
        9m^2-60m+96, & \text{if $\abs J = 1$,}\\
        n+\frac{3m^3-45m^2+172m-212}2, & \text{if $\abs J = 2$,}
    \end{cases}
    &
    p_{\set1\cup J} &= \begin{cases}
        0, & \text{if $\abs J=0$,}\\
        n-3m^2+5m+2, & \text{if $\abs J = 1$,}\\
        \ups(m-6,n-4-Q(m)) & \text{if $\abs J = 2$.}\\
        \quad{}={}\frac{mn-5n-3m^3+20m^2-23m-8}2,\span\omit
    \end{cases}
\end{align*}
Note that $p_{\set1}=0$ is \emph{not} due to an erasure of irrelevant points but rather just a curious arithmetic occurrence.
}

\famitem D1{
For any $(m,n)$ with $m\geq6$, $n\geq Q(m)+5$, we will define $D_1(m,n)$ with $k=4$ coordinate subvarieties through (again, $J$ is a subset of $\set{2,3}$):
\begin{align*}
    \tilde u_{\emptyset} &= m-5, & \tilde u_{\set1} &= \tilde u_{\set2} = \tilde u_{\set3} = 2,\\
    \tilde u_I &= 0\quad\text{for all other $I$},\span\omit\\
    \tilde v_J &= \begin{cases}
        n-5-Q(m), & \text{if $\abs J=0$,}\\
        2, & \text{if $\abs J = 1$,}\\
        0, & \text{if $\abs J = 2$,}
    \end{cases}
    &
    \tilde v_{\set1\cup J} &= \begin{cases}
        Q(m-4), & \text{if $\abs J = 0$,}\\
        12m-48, & \text{if $\abs J = 1$,}\\
        24, & \text{if $\abs J = 2$,}
    \end{cases}\\
    \tilde v_{\set4} &= 2, & \tilde v_{I} &= 0\quad\text{for all other $I\ni 4$,}\\
    p_J &= \begin{cases}
        0, & \text{if $\abs J<2$,}\\
        2, & \text{if $\abs J = 2$,}
    \end{cases}
    &
    p_{\set1\cup J} &= \begin{cases}
        0, & \text{if $\abs J=0$,}\\
        2, & \text{if $\abs J = 1$,}\\
        m-5, & \text{if $\abs J = 2$,}\\
    \end{cases}\\
    p_{\set{4}} &= 36m-84, & p_I &= 0\quad\text{for all $I\supsetneq\set4$.}
\end{align*}
This is both an $(m-2,n-\ell(m))$-inductant of $C_1$ (with $L_3$ being the newly added subvariety in the inductant construction) and, due to Remark~\ref{rmrk:commuting-inductants}, also an $(m,n-2)$-inductant of $D_0$ (with $L_4$ being the newly added subvariety). Erasure of points is responsible for $p_I=0$ for $I\supsetneq\set4$.
}

\famitem E0{
For any $(m,n)$ with $m\geq8$, $n\geq\max\set{3m^2-5m-2, Q(m)+5}$ (note that this equals $3m^2-5m-2$ for $m\geq7$), we define $E_0(m,n)$ with $k=4$ coordinate subvarieties through the following choice of parameters, where we denote by $J$ a subset of $\set{2,3,4}$:
\begin{align*}
    \tilde u_I &= \begin{cases}
        m-7, & \text{if $\abs I = 0$,}\\
        2, & \text{if $\abs I = 1$,}\\
        0, & \text{if $\abs I > 1$,}
    \end{cases}\span\omit\\
    \tilde v_J &= \begin{cases}
        n-5-Q(m), & \text{if $\abs J=0$,}\\
        2, & \text{if $\abs J=1$,}\\
        0, & \text{if $\abs J>1$,}
    \end{cases}
    &
    \tilde v_{\set1\cup J} &= \begin{cases}
        Q(m-6), & \text{if $\abs J=0$,}\\
        12m-72, & \text{if $\abs J=1$,}\\
        24, & \text{if $\abs J=2$,}\\
        0, & \text{if $\abs J=3$,}
    \end{cases}\\
    p_J &= \begin{cases}
        72, & \text{if $\abs J=0$,}\\
        36m-156, & \text{if $\abs J=1$,}\\
        9m^2-96m+252, & \text{if $\abs J=2$,}\\
        0, & \text{if $\abs J=3$,}
    \end{cases}
    &
    p_{\set1\cup J} &= \begin{cases}
        0, & \text{if $\abs J=0$,}\\
        0, & \text{if $\abs J=1$,}\\
        n-3m^2+5m+2, & \text{if $\abs J=2$,}\\
        0, & \text{if $\abs J=3$.}
    \end{cases}
\end{align*}
This is an $(m-2,n-\ell(m))$-inductant of $D_0$ on nice $(m,n)$ where it is defined, though we define $E_0$ also on ugly $(m,n)$ satisfying the bounds above as well.
Note that $p_{\set{2,3,4}} = p_{\set{1,2,3,4}} = 0$ is due to erasure of irrelevant points.
}

\famitem E1{
For any $(m,n)$ with $m\geq 8$, $n\geq Q(m)+7$, we define $E_1(m,n)$ with $k=5$ coordinate subvarieties through ($J$ being a subset of $\set{2,3,4}$)
\begin{align*}
    \tilde u_\emptyset &= m-7, & \tilde u_{\set1} &= u_{\set2} = u_{\set3} = u_{\set4} = 2,\\
    \tilde u_I &= 0\quad\text{for all other $I$},\span\omit\\
    \tilde v_J &= \begin{cases}
        n-7-Q(m), & \text{if $\abs J=0$,}\\
        2, & \text{if $\abs J=1$,}\\
        0, & \text{if $\abs J>1$,}
    \end{cases}
    &
    \tilde v_{\set1\cup J} &= \begin{cases}
        Q(m-6), & \text{if $\abs J=0$,}\\
        12m-72, & \text{if $\abs J=1$,}\\
        24, & \text{if $\abs J=2$,}\\
        0, & \text{if $\abs J=3$,}
    \end{cases}\\
    \tilde v_{\set5} &= 2, & \tilde v_{I} &= 0\quad\text{for all other $I\ni 5$},\\
    p_{\set5} &= 72, & p_{\set{1,2,3}} &= p_{\set{1,2,3}} = p_{\set{1,3,4}} = 2,\\
    p_{I} &= 0\quad\text{for all other $I$.}\span\omit
\end{align*}
Note that most of the $p_I$'s have vanished due to erasure of irrelevant points at this stage. This $E_1$ is both an $(m-2,n-\ell(m))$-inductant of $D_1$ and due to Remark~\ref{rmrk:commuting-inductants} also an $(m,n-2)$-inductant of $E_0$.
}

\famitem F0{
For any $(m,n)$ with $m\geq10$, $n\geq Q(m)+7$, we define $F_0(m,n)$ with $k=5$ coordinate subvarieties through ($J$ being a subset of $\set{2,3,4,5}$):
\begin{align*}
    \tilde u_I &= \begin{cases}
        m-9, & \text{if $\abs I = 0$,}\\
        2, & \text{if $\abs I = 1$,}\\
        0, & \text{if $\abs I > 1$,}
    \end{cases}\span\omit\\
    \tilde v_J &= \begin{cases}
        n-7-Q(m), & \text{if $\abs J=0$,}\\
        2, & \text{if $\abs J=1$,}\\
        0, & \text{if $\abs J>1$,}
    \end{cases}
    &
    \tilde v_{\set1\cup J} &= \begin{cases}
        Q(m-8), & \text{if $\abs J=0$,}\\
        12m-96, & \text{if $\abs J=1$,}\\
        24, & \text{if $\abs J=2$,}\\
        0, & \text{if $\abs J>2$,}
    \end{cases}\\
    p_J &= \begin{cases}
        0, & \text{if $\abs J=0$,}\\
        72, & \text{if $\abs J=1$,}\\
        36m-228, & \text{if $\abs J=2$,}\\
        0, & \text{if $\abs J>2$,}
    \end{cases}
    &
    p_{\set1\cup J} &= 0.
\end{align*}
This is an $(m-2,n-\ell(m))$-inductant of $E_0$.
}

\famitem G0{
For any $(m,n)$ with $m\geq12$, $n\geq Q(m)+9$, we define $G_0(m,n)$ with $k=6$ coordinate subvarieties through ($J$ being a subset of $\set{2,\dots,6}$):
\begin{align*}
    \tilde u_I &= \begin{cases}
        m-11, & \text{if $\abs I = 0$,}\\
        2, & \text{if $\abs I = 1$,}\\
        0, & \text{if $\abs I > 1$,}
    \end{cases}\span\omit\\
    \tilde v_J &= \begin{cases}
        n-9-Q(m), & \text{if $\abs J=0$,}\\
        2, & \text{if $\abs J=1$,}\\
        0, & \text{if $\abs J>1$,}
    \end{cases}
    &
    \tilde v_{\set1\cup J} &= \begin{cases}
        Q(m-10), & \text{if $\abs J=0$,}\\
        12m-120, & \text{if $\abs J=1$,}\\
        24, & \text{if $\abs J=2$,}\\
        0, & \text{if $\abs J>2$,}
    \end{cases}\\
    p_J &= \begin{cases}
        72, & \text{if $\abs J=2$,}\\
        0, & \text{otherwise,}
    \end{cases}
    &
    p_{\set1\cup J} &= 0.
\end{align*}
This is an $(m-2,n-\ell(m))$-inductant of $F_0$.
}
\end{itemize}
Let us also briefly comment on these configurations in relation to the conditions \eqref{eq:coordsubvar-nonempty} and \eqref{eq:coordconfig-sanity}. Firstly, in all of the configurations listed above, our bounds ensure $\tilde u_{\emptyset}>0$. Since $\emptyset$ contains no elements and is disjoint from any set, $\tilde u_\emptyset$ is always included in the sums in question, so the desired inequalities hold.  For the conditions on $\tilde v_I$, the situation is as follows:
\begin{itemize}
    \item For $B_0$, $C_0$, $D_0$, the bound on $n$ ensures that $v_\emptyset$ is non-zero, so the same reasoning as above applies.
    \item For all of the other configurations, notice that:
    \begin{enumerate}[label={(\roman*)}]
    \item They have at least two subvarieties.
    \item The parameter $p_{\set{1,\dots,k}}$ is $0$. This means no double points are constrained to the intersection of all the subvarieties.
    \item The values at singletons $\tilde v_{\set i}$ are all positive. This is because they are all either constant $2$ or some value $Q(m')$ for an $m'\geq2$.
    \end{enumerate}
    We then notice that \eqref{eq:coordsubvar-nonempty} holds because the sum for any $t$ includes some singleton $\tilde v_{\set i}$, $i\neq t$, because $k\geq2$. Similarly, \eqref{eq:coordconfig-sanity} holds, because any $I$ with $p_I>0$ has an element $j\notin I$, so then $\tilde v_{\set j}>0$ is included in the sum, ensuring its positivity.
\end{itemize}

\begin{remark}
Let us comment on the choice of the polynomials $Q$ and $\ell$ for the step sizes in our inductive strategy. First, $Q(m)$ is chosen with the motivation that $\vdim B_0(m,n) = \vdim A_0(m,n) - \vdim A_0(m-2, n-Q(m))$ be non-positive, which is necessary for applying Lemma~\ref{lem:inductant-induction}, but preferably with a small absolute value. Since we observed our $\vdim B_0$ to be constantly $-1$, one could hope to decrease $Q(m)$ by one to $3m^2-6m+1$, making $B_0$ equiabundant, with the goal of slightly enlarging the set of $(m,n)$ where we prove non-defectivity. However, $3m^2-6m+1$ is always odd, hence for even $m$, we would end up connecting nice cases with ugly ones, breaking the induction, because the ugly cases of $A_0(m,n)$ have a higher virtual dimension in terms of $n$ compared to nearby nice cases. For the sake of keeping the presentation simpler, we opted to sacrifice the minute improvements to be had in constructing separate inductions for even and odd $m$'s in the nice case, leading to our choice of $Q(m) = 3m^2-6m+2$.

In the subsequent steps, the choice of $\ell(m)$ stems from keeping the $\tilde v$ parameters non-negative. Observe that $\tilde v_{\set2}$ in $C_0(m,n)$ becomes a constant $2$. Interestingly, we could decrease $\ell(m)$ to $12m-24$, resulting in $\tilde v_{\set2}=0$ still being non-negative. However, testing the resulting configurations computationally, we found them to be probably defective via Algorithm~\ref{alg:coordconfig} (see Section~\ref{sec:computations}). We do not have a satisfying explanation for this defectivity. After this, we tried increasing the constant term by two (to preserve parity), leading to our choice of $\ell(m)=12m-22$, which yields a successful inductive strategy. Again, some minor improvements could perhaps be extracted by choosing tighter steps sizes in some of the further parts of Figure~\ref{fig:nicestrategy}, but we opted instead to minimize, via Remark~\ref{rmrk:commuting-inductants}, the number of unique families we consider.
\end{remark}

To prove non-defectiveness on large sets of $(m,n)$ of the families we just defined, we will use the following base cases:
\begin{prop}
    \label{prop:nicecertificates}
    The following are non-defective and superabundant (or even equiabundant):
    \begin{gather*}
        A_0(3,8),\quad A_0(3,9),\quad A_0(4,26),\\
        B_0(2,4),\quad B_0(3,12),\quad B_0(3,13),\quad B_1(2,4),\quad B_1(3,14),\quad B_2(2,5),\\
        C_0(4,30),\quad C_0(5,50),\quad C_0(5,51),\quad C_1(4,30),\quad C_1(5,52),\\
        D_0(6,82),\quad D_0(7,111),\quad D_0(7,112),\quad D_1(6,80),\quad D_1(7,112),\\
        E_0(8,152),\quad E_0(9,196),\quad E_0(9,197),\quad E_1(8,154),\\
        F_0(10,250),\quad F_0(11,306),\quad G_0(12,372).
    \end{gather*}
\end{prop}
\begin{proof}
    Verified computationally (see Section~\ref{sec:computations}).
\end{proof}

To prove results about these families inductively, we essentially have to proceed in an order opposite to that of their construction. I.e. in terms of the diagram in Figure~\ref{fig:nicestrategy}, we start from the right- and topmost nodes and work our way to the left and to the bottom. The proofs essentially consist only of verifying the inductant constructions, observing that some further inductants become trivial and managing the bounds on $m$, $n$. Therefore, we shall be brief in latter propositions where the same ideas are repeated.

\begin{prop}
    \label{prop:G0}
    For any $(m,n)$ with $m\geq 12$, $n\geq 3m^2-5m$, the configuration $G_0(m,n)$ is non-defective.
\end{prop}
\begin{proof}
    First we prove that non-defectivity of $G_0(m,n-1)$ implies non-defectivity of $G_0(m,n)$. By Lemma~\ref{lem:inductant-induction}, we only need to prove non-defectivity (and the correct abundancy) of the $(m,n-1)$-inductant of $G_0$, let us call it $S$, and let $L_7$ be the new subvariety added in the inductant construction. Note that of all sets $I\ni 7$, only $I=\set 7$ will have non-zero $\tilde v_I$, namely $\tilde v_{\set7}=1$.
    Thus since the numbers of double points in $G_0$ were already constant and already constrained to some subvariety, all points will become irrelevant, so $S(m,n)$ has no points.

    Let us argue that no bidegree $(1,2)$ monomial vanishes on all $7$ coordinate subvarieties of $S$. If $L_7$ is the newly added subvariety, only one $y_j$ vanishes on it, so any monomial vanishing on all seven subvarieties must be divisible by it.
    Each $x_i$ vanishes on at most one of $L_1,\dots,L_6$ and never on $L_7$. Lastly, any $y_j$ vanishes on at most two of $L_2,\dots,L_6$ (and possibly also on $L_1$), hence no $x_iy_{j_1}y_{j_2}$ can vanish on all seven of the subvarieties. Therefore we see that both $\vdim S(m,n)=0$ and also always $\mathcal I_{S(m,n)} = 0$. Hence $S(m,n)$ is equiabundant and non-defective. From this, Lemma~\ref{lem:inductant-induction} gives us the inductive step
    \[
        \text{$G_0(m,n-1)$ non-defective} \implies \text{$G_0(m,n)$ non-defective}.
    \]

    Second, we prove that non-defectivity of $G_0(m-1,n-6m+9)$ implies non-defectivity of $G_0(m,n)$. Again we consider the corresponding inductant, let us call it $S'$, and again, we wish to prove that it has no points and no bidegree $(1,2)$ monomial vanish on all of its subvarieties simultaneously. The points all becoming irrelevant in the inductant construction follows for the exact same reasons as with $S$.

    For the monomials, again, each $x_i$ vanishes on at most one of $L_1,\dots,L_7$. Applying the relations \eqref{eq:inductant-construction}, we also see that each $y_j$ may vanish on at most two of $L_2,\dots,L_7$ (and also additionally on $L_1$), meaning a bidegree $(1,2)$ monomial $x_iy_{j_1}y_{j_2}$ can vanish on all six $L_2,\dots, L_7$, and thus much less on all seven $L_1,\dots,L_7$. Therefore the corresponding ideal vanishes, so $S'(m,n)$ is indeed equiabundant and non-defective, yielding the inductive step
    \[
        \text{$G_0(m-1,n-6m+9)$ non-defective} \implies \text{$G_0(m,n)$ non-defective}
    \]
    by Lemma~\ref{lem:inductant-induction} once more.

    Finally, we observe that we have non-defectivity of $G_0(12,372)$ by Proposition~\ref{prop:nicecertificates} and that $372 = 3\cdot 12^2-5\cdot12$. Further, $3m^2-5m + (6m-9) = 3(m+1)^2-5(m+1)$, so combining the two inductions above, the full proposition follows from non-defectivity of $G_0(12,372)$.
\end{proof}

\begin{prop}
    \label{prop:F0}
    For any $(m,n)$ with $m\geq 10$, $n\geq 3m^2-5m$, the configuration $F_0(m,n)$ is equiabundant and non-defective.
\end{prop}
\begin{proof}
    Let $S$ be the $(m,n-1)$ inductant of $F_0$. By same reasoning as in the proof of the previous proposition, let us see that all points become irrelevant in this inductant: the functions $p_I$ in $F_0$ were constant with respect to $n$. Let us also see that the ideal of the subvarieties vanishes in bidegree $(1,2)$. There are $7$ subvarieties, $L_7$ requires its sole variable $y_j$ that is shared by no other. Any $x_i$ can vanish on at most one subvariety and among $L_1,\dots,L_6$, any $y_j$ vanishes at most on three. Hence vanishing on all $7$ cannot be achieved. This equiabundancy and non-defectivity gives an induction
    \[
        \text{$F_0(m,n-1)$ non-defective} \implies \text{$F_0(m,n)$ non-defective}.
    \]
    Further, we know that $G_0$, which is the $(m-2,n-\ell(m))$-inductant of $F_0$, is non-defective by Proposition~\ref{prop:G0}. This gives the induction
    \[
        \text{$F_0(m-2,n-\ell(m))$ non-defective} \implies \text{$F_0(m,n)$ non-defective}
    \]
    for $m\geq 12$, $n\geq 3m^2-5m$. Note that $3m^2-5m$ evaluated at $10$ and $11$ is $250$ and $308$ respectively. By Proposition~\ref{prop:nicecertificates}, we know non-defectivity of $F_0(10,250)$ and $F_0(11,306)$. Through the inductions above, we perform steps $F_0(11,306)\implies F_0(11,307) \implies F_0(11,308)$. Then $F_0(10,250)$ and $F_0(11,308)$ prove non-defectivity of all $F_0(m,3m^2-5m)$, $m\geq10$, because $12m-22$ is the difference $3m^2-5m - 3(m-2)^2+5(m-2)$, and finally, $F_0(m,n-1)\implies F_0(m,n)$ extends this to all $F_0(m,n)$, $m\geq10$, $n\geq 3m^2-5m$.
\end{proof}

\begin{prop}
    \label{prop:E1}
    For any $(m,n)$ with $m\geq 8$, $n\geq 3m^2-5m+2$, the configuration $E_1(m,n)$ is equiabundant and non-defective.
\end{prop}
\begin{proof}
    We again establish two kinds of induction steps through two inductants that we show become trivial. First, let $S$ be the $(m,n-1)$-inductant of $E_1$. Then all points become irrelevant, since all $p_I$ were already constant in $E_1$. Further, no bidegree $(1,2)$ monomial vanish on all $6$ subvarieties of $S$, since two different $y_j$'s are needed to vanish on the two subvarieties coming from $(m,n-1)$ and $(m,n-2)$ inductant constructions, and subsequently any variable $x_i$ vanishes  at most on one of the remaining four subvarieties. This establishes the induction
    \[
        \text{$E_1(m,n-1)$ non-defective} \implies \text{$E_1(m,n)$ non-defective}.
    \]

    Then, let $S'$ be the $(m-1,n-6m+9)$-inductant of $E_1$. The points again become irrelevant, because all $p_I$ were constant in $E_1$. To see that no bidegree $(1,2)$ monomials vanish on all $7$ subvarieties of $S'$, we argue as follows. One subvariety comes from the $(m,n-2)$-inductant construction that created $E_1$ from $E_0$, this only has two $y_j$ equations that are shared by none other. Any $x_i$ only vanishes on at most one of the subvarieties and any $y_j$ vanishes on at most three. So altogether vanishing on all $7$ cannot be achieved. This establishes
    \[
        \text{$E_1(m-1,n-6m+9)$ non-defective} \implies \text{$E_1(m,n)$ non-defective}.
    \]

    Then starting from non-defectivity of $E_1(8,154)$ which we have by Proposition~\ref{prop:nicecertificates}, the above two induction extend it to all $E_1(m,n)$ for $m\geq8$, $n\geq 3m^2-5m+2$, since $3m^2-5m+2$ evaluates to $154$ at $m=8$.
\end{proof}

\begin{prop}
    \label{prop:E0}
    For any nice $(m,n)$ with $m\geq 8$, $n\geq 3m^2-5m$, the configuration $E_0(m,n)$ is equiabundant and non-defective.
\end{prop}
\begin{proof}
    By Propositions~\ref{prop:E1} and \ref{prop:F0}, we have induction steps
    \begin{align*}
        \text{$E_0(m,n-2)$ non-defective} &\implies \text{$E_0(m,n)$ non-defective}\quad\text{(for $m\geq8$, $n\geq 3m^2-5m+2$),}\\
        \text{$E_0(m-2,n-\ell(m))$ non-defective} &\implies \text{$E_0(m,n)$ non-defective}\quad\text{(for $m\geq 10$, $n\geq3m^2-5m$)},
    \end{align*}
    for nice $(m,n)$.
    By Proposition~\ref{prop:nicecertificates}, we have non-defectivity of $E_0(8,152)$, $E_0(9,196)$ and $E_0(197)$. Note that $3m^2-5m$ evaluated at $m=8$ and $m=9$ gives $152$ and $198$ respectively. Hence we use the $E_0(m,n-2)\implies E_0(m,n)$ induction step to first obtain non-defectivity of $E_0(9,198)$ and $E_0(9,199)$, and then the two induction steps extend from $E_0(8,152)$, $E_0(9,198)$, $E_0(9,199)$ to all nice $(m,n)$ with $m\geq 8$, $n\geq 3m^2-5m$.
\end{proof}

\begin{prop}
    \label{prop:D1}
    For any $(m,n)$ with $m\geq 6$, $n\geq 3m^2-5m+2$, the configuration $D_1(m,n)$ is equiabundant and non-defective.
\end{prop}
\begin{proof}
    Let $S$ be the $(m,n-1)$-inductant of $D_1$. All $p_I$ in $D_1$ do not depend on $n$ and $p_{\emptyset}=0$, hence all points will get specialized to an intersection of the new subvariety $L$ with some of the previous, and thus become irrelevant, since the new subvariety will not share any equations with any of the other. No bidegree $(1,2)$ monomial will vanish on all of the $5$ subvarieties of $S'$, since one $y_j$ must be spent to vanish on $L$ and is not shared with any of the others. Another $y_j$ is needed to vanish on the subvariety coming from the $(m,n-2)$ inductant construction that created $D_1$ from $D_0$ and any $x_i$ vanishes on at most one of them. Hence no bidegree $(1,2)$ can vanish on all five. This establishes the induction
    \[
        \text{$D_1(m,n-1)$ non-defective} \implies \text{$D_1(m,n)$ non-defective}.
    \]
    From Proposition~\ref{prop:E1}, we also have the induction
    \[
        \text{$D_1(m-2,n-\ell(m))$ non-defective} \implies \text{$D_1(m,n)$ non-defective (for $m\geq 8$, $n\geq 3m^2-5m+2$)}.
    \]

    By Proposition~\ref{prop:nicecertificates}, we have non-defectivity of $D_1(6,80)$ and $D_1(7,112)$. The latter one implies the non-defectivity of $D_1(7,113)$ and in turn $D_1(7,114)$, and we notice that $3m^2-5m+2$ evaluated at $m=6$ and $m=7$ yields $80$ and $114$ respectively. Then the two inductions above extend non-defectivity of $D_1(6,80)$ and $D_1(7,114)$ to all $D_1(m,n)$ for $m\geq 6$, $n\geq 3m^2-5m+2$.
\end{proof}

\begin{prop}
    \label{prop:D0}
    For any nice $(m,n)$ with $m\geq 6$, $n\geq 3m^2-5m+4$, the configuration $D_0(m,n)$ is equiabundant and non-defective.
\end{prop}
\begin{proof}
    By Propositions~\ref{prop:D1} and \ref{prop:E0}, we have induction steps
    \begin{align*}
        \text{$D_0(m,n-2)$ non-defective} &\implies \text{$D_0(m,n)$ non-defective}\quad\text{(for $m\geq6$, $n\geq 3m^2-5m+2$),}\\
        \text{$D_0(m-2,n-\ell(m))$ non-defective} &\implies \text{$D_0(m,n)$ non-defective}\quad\text{(for $m\geq 8$, $n\geq3m^2-5m$)},
    \end{align*}
    for nice $(m,n)$. Note that $3m^2-5m+4$ evaluates to $82$ and $116$ respectively at $m=6$ and $m=7$. By Proposition~\ref{prop:nicecertificates}, we have non-defectivity of $D_0(6,82)$, $D_0(7,111)$ and $D_0(112)$. We use the first kind of induction to prove $D_0(7,116)$ and $D_0(7,117)$ first. Subsequently, the two inductions extend the three base cases $D_0(6,82)$, $D_0(7,116)$, $D_0(7,117)$ to all nice $(m,n)$ with $m\geq 6$, $n\geq 3m^2-5m+4$.
\end{proof}

\begin{prop}
    \label{prop:C1}
    For any $(m,n)$ with $m\geq 4$, $n\geq 3m^2-5m+2$, the configuration $C_1(m,n)$ is equiabundant and non-defective.
\end{prop}
\begin{proof}
    Let $S$ be the $(m,n-1)$-inductant of $C_1$. All $p_I$ in $C_1$ are constant with respect to $n$ and $p_{\emptyset}$ is zero, hence in the inductant construction, all points will be specialized to an intersection of the new subvariety $L$ with an existing coordinate subvariety. Since $L$ will be defined by a single equation, not shared with any other subvariety, all of these points will become irrelevant. Further, no bidegree $(1,2)$ monomial will vanish on all $4$ subvarieties of $C_1$: one variable $y_j$ is needed just to vanish on $L$ (and it does not vanish on any of the others), another $y_j$ is needed to vanish on the subvariety coming from the $(m,n-2)$ inductant construction that created $C_1$ from $C_0$ and any $x_i$ may only vanish on at most one of the remaining subvarieties. Hence $S$ is equiabundant and non-defective everywhere it is defined, establishing the induction
    \[
        \text{$C_1(m,n-1)$ non-defective} \implies \text{$C_1(m,n)$ non-defective}.
    \]
    By Proposition~\ref{prop:D1}, we also have the induction
    \[
        \text{$C_1(m-2,n-\ell(m))$ non-defective} \implies \text{$C_1(m,n)$ non-defective (for $m\geq 6$, $n\geq 3m^2-5m+2$)}.
    \]

    These two inductions combined will then extend the base cases $C_1(4,30)$, $C_1(3,52)$ (note that $3m^2-5m+2$ evaluated to $30$ and $52$ at $m=4$ and $m=5$ respectively), which are non-defective by Proposition~\ref{prop:nicecertificates}, to all $(m,n)$ with $m\geq4$, $n\geq 3m^2-5m+2$.
\end{proof}

\begin{prop}
    \label{prop:C0}
    For any nice $(m,n)$ with $m\geq 4$, $n\geq 3m^2-5m+4$, the configuration $C_0(m,n)$ is equiabundant and non-defective.
\end{prop}
\begin{proof}
    By Propositions~\ref{prop:C1} and \ref{prop:C0}, we have induction steps
    \begin{align*}
        \text{$C_0(m,n-2)$ non-defective} &\implies \text{$D_0(m,n)$ non-defective}\quad\text{(for $m\geq4$, $n\geq 3m^2-5m+2$),}\\
        \text{$C_0(m-2,n-\ell(m))$ non-defective} &\implies \text{$D_0(m,n)$ non-defective}\quad\text{(for $m\geq 6$, $n\geq3m^2-5m+4$)},
    \end{align*}
    for nice $(m,n)$.
    Note that $3m^2-5m+4$ evaluated at $m=4$ and $m=5$ yields $32$ and $54$ respectively. By Proposition~\ref{prop:nicecertificates}, we have non-defectivity of $C_0(4,30)$, $C_0(5,50)$ and $C_0(5,51)$. By the first kind of induction, we first obtain that $C_0(4,32)$, $C_0(5,54)$ and $C_0(5,55)$ are non-defective, and subsequently, the two induction extend this non-defectivity to all $C_0(m,n)$ for nice $(m,n)$ with $m\geq4$, $n\geq 3m^2-5m+4$.
\end{proof}

\begin{prop}
    \label{prop:B2}
    For any $(m,n)$ with $m\geq 2$, $n\geq 3m^2-5m+3$, the configuration $B_2(m,n)$ is equiabundant and non-defective.
\end{prop}
\begin{proof}
    Let $S$ be the $(m,n-1)$-inductant of $B_2$. Since all $p_I$ were constant in $B_2$ and $p_\emptyset=0$, all points will get specialized to an intersection of the new one-codimensional coordinate subvariety $L$ with one of the existing subvarieties, and thus become irrelevant. On the side of monomials, $S$ will have three subvarieties that only have $y_j$ equations and that do not share any of their equations with other subvarieties, hence any monomial vanishing on all three of them would need to have degree at least $3$ in the variables $y_j$. Thus $S$ is always equiabundant and non-defective, which establishes the induction
    \[
        \text{$B_2(m,n-1)$ non-defective} \implies \text{$B_2(m,n)$ non-defective}.
    \]

    Now let $S'$ be the $(m-1,n-6m+9)$-inductant of $B_2$. Due to Remark~\ref{rmrk:commuting-inductants}, we have actually already investigated this family in the proof of Proposition~\ref{prop:C1}, hence we immediately have
    \[
        \text{$B_2(m-1,n-6m+9)$ non-defective} \implies \text{$B_2(m,n)$ non-defective}.
    \]

    By Proposition~\ref{prop:nicecertificates}, we know that $B_2(2,5)$ is non-defective. Since $3m^2-5m+3$ evaluates to $5$ at $m=2$, the two inductions above extend this base case to all $(m,n)$ with $m\geq2$, $n\geq 3m^2-5m+3$.
\end{proof}

\begin{prop}
    \label{prop:B1}
    For any nice $(m,n)$ with $m\geq 2$, $n\geq 3m^2-5m+2$, the configuration $B_1(m,n)$ is equiabundant and non-defective.
\end{prop}
\begin{proof}
    By Propositions~\ref{prop:B2} and \ref{prop:C1}, we have induction steps
    \begin{align*}
        \text{$B_1(m,n-1)$ non-defective} &\implies \text{$B_1(m,n)$ non-defective}\quad\text{(for $m\geq2$, $n\geq 3m^2-5m+3$),}\\
        \text{$B_1(m-2,n-\ell(m))$ non-defective} &\implies \text{$B_1(m,n)$ non-defective}\quad\text{(for $m\geq 4$, $n\geq3m^2-5m+2$)}.
    \end{align*}
    Note that $3m^2-5m+2$ evaluated at $m=2$ and $m=3$ yields $4$ and $14$ respectively. Hence the non-defectivity of $B_1(2,4)$ and $B_1(3,14)$ which we have by Proposition~\ref{prop:nicecertificates} extends via the two inductions above to all $B_1(m,n)$ for $m\geq 2$, $n\geq 3m^2-5m+2$.
\end{proof}

\begin{prop}
    \label{prop:B0}
    For any nice $(m,n)$ with $m\geq 2$, $n\geq 3m^2-5m+4$, the configuration $B_0(m,n)$ is superabundant and non-defective.
\end{prop}
\begin{proof}
    By Propositions~\ref{prop:B1} and \ref{prop:C0}, we have induction steps
    \begin{align*}
        \text{$B_0(m,n-2)$ non-defective} &\implies \text{$B_0(m,n)$ non-defective}\quad\text{(for $m\geq2$, $n\geq 3m^2-5m+2$),}\\
        \text{$B_0(m-2,n-\ell(m))$ non-defective} &\implies \text{$B_0(m,n)$ non-defective}\quad\text{(for $m\geq 4$, $n\geq3m^2-5m+4$)},
    \end{align*}
    for nice $(m,n)$.
    Note that $3m^2-5m+4$ evaluated at $m=2$ and $m=3$ yields $6$ and $16$ respectively, whilst by Proposition~\ref{prop:nicecertificates}, we have non-defectivity of $B_0(2,4)$, $B_0(3,12)$ and $B_0(3,13)$. First we use the first kind of induction to obtain $B_0(2,6)$, $B_0(3,16)$, $B_0(3,17)$ from these, and then the two inductions extend the non-defectivity to all nice $(m,n)$ with $m\geq2$, $n\geq 3m^2-5m+4$.
\end{proof}

\begin{theorem}
    \label{thrm:A0nice}
    For any nice $(m,n)$ with $m\geq3$, $n\geq \ceil{\frac{m^3-2m-5}2}$, the configuration $A_0(m,n)$ is superabundant and non-defective.
\end{theorem}
\begin{proof}
    Note that $\ceil{\frac{m^3-2m-5}2} = \frac{m^3-2m-5}2$ for odd $m$ and $\ceil{\frac{m^3-2m-5}2} = \frac{m^3-2m-4}2$ for even $m$. In either case,
    $3\binom{m+1}3 - \zav{\ceil{\frac{m^3-2m-5}2} + m+1} = -\ceil{\frac{m-3}2}$ is non-positive for $m\geq3$, so $A_0(m,n)$ will be superabundant for all nice $(m,n)$ satisfying the bounds. Note further that
    \[
        \ceil{\frac{m^3-2m-5}2}  - \ceil{\frac{(m-2)^3-2(m-2)-5}2} = Q(m).
    \]

    Observe that for $m\geq5$, we have $\ceil{\frac{m^3-2m-5}2}\geq 3m^2-5m+4$, hence by Proposition~\ref{prop:B0} we have the induction step
    \[
        \text{$A_0(m-2,n-Q(m))$ non-defective} \implies \text{$A_0(m,n)$ non-defective}\quad\text{(for $m\geq5$, $\textstyle n\geq \ceil{\frac{m^3-2m-5}2}$)}.
    \]
    We also have $\ceil{\frac{m^3-2m-5}2}+2\geq m$ for all $m\geq3$, hence Proposition~\ref{prop:A1} applies and gives us the induction step
    \[
        \text{$A_0(m,n-2)$ non-defective} \implies \text{$A_0(m,n)$ non-defective}\quad\text{(for $m\geq3$, $\textstyle n\geq \ceil{\frac{m^3-2m-5}2}+2$)}.
    \]

    By Proposition~\ref{prop:nicecertificates}, we have non-defectivity of $A_0(3,8)$, $A_0(3,9)$ and $A_0(4,26)$; note that $\ceil{\frac{m^3-2m-5}2}$ evaluates to $8$ and $26$ at $m=3$ and $m=4$ respectively. Thus the two induction steps described above then extend these base cases to non-defectivity of $A_0(m,n)$ for all nice $(m,n)$ with $m\geq3$, $n\geq\ceil{\frac{m^3-2m-5}2}$.
\end{proof}

\begin{remark}
    \label{rmrk:vertical-alternative}
    Proposition~\ref{prop:A1} could alternatively be proved through further inductants. There are likely many different strategies to facilitate this, we sketch one in Figure~\ref{fig:verticalstrategy} that we believe works.
    We do not wish to encumber this article with pursuing it in detail, rather, our purpose in this remark is to note that Proposition~\ref{prop:A1} is not a qualitatively different tool from the rest of the inductant machinery of this subsection.

\begin{figure}[tb]
    \centering
    \begin{tikzpicture}[n/.style = {draw, circle, inner sep=2pt}, a/.style = {-{Latex[scale=1]}, shorten <=2pt, shorten >=2pt}, scale=2.35]
        \node[n] (A0) at (0,0) {$A_0$};
        \node[n] (A1) at (0,1) {$A_1$};
        \node[n] (A2) at (0,2) {$A_2$};
        \node[n] (A3) at (1,1) {$A_3$};
        \node[n] (A4) at (1,2) {$A_4$};
        \node[n] (A5) at (2,1) {$A_5$};

        \draw[a](A1) -- (A0) node[midway, left=.1cm] {\tinylabel{$(m,n-2)$}};
        \draw[a](A2) -- (A1) node[midway, left=.1cm] {\tinylabel{$(m,n-2)$}};
        \draw[a](A4) -- (A3) node[midway, right=.1cm] {\tinylabel{$(m,n-2)$}};

        \draw[a](A3) -- (A1) node[midway, below=.3cm] {\tinylabel{$(m-2,n-2)$}};
        \draw[a](A4) -- (A2) node[midway, above=.3cm] {\tinylabel{$(m-2,n-2)$}};
        \draw[a](A5) -- (A3) node[midway, below=.3cm] {\tinylabel{$(m-1,n-1)$}};
    \end{tikzpicture}
    \caption{Families of coordinate configurations in an inductive strategy to prove non-defectivity of $A_1$. An arrow labeled with $(M,N)$ denotes that a family is an $(M,N)$-inductant of the family being pointed to. These inductant constructions are sufficient to define all $A_2,\dots, A_6$; we defined $A_1$ explicitly in the paragraph preceding Proposition~\ref{prop:A1}. The reader may notice that $A_2(m,n)$ is actually the configuration from Example~\ref{ex:basic}.}
    \label{fig:verticalstrategy}
\end{figure}
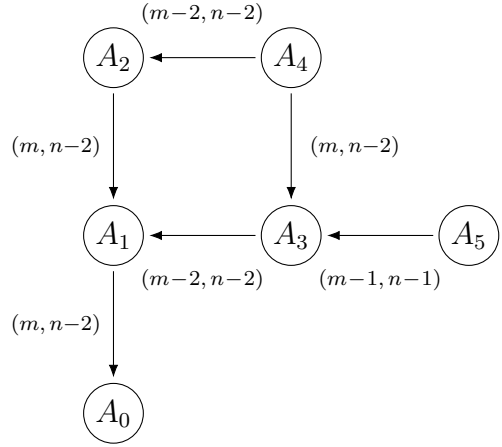
\end{remark}

\subsection{The ugly case}
\label{subsec:ugly}

Since we already have non-defectivity of $A_0(m,n)$ for nice cases with $n\gg m$, we will now try to prove non-defectivity of some of the ugly cases of $A_0(m,n)$ with an induction step from some nice $(m,n')$. Comparing virtual dimensions in the ugly and nice cases and solving
\[
     \frac{n+m+1}2- 3\binom{m+1}3  \geq (n'+m+1) - 3\binom{m+1}3
\]
for $n'$ suggests we should choose $n'\leq \frac{n-m-1}{2}$. We need to ensure that $n'$ is even though, so the largest possible $n'$ then becomes $\frac{n-m-\epsilon_{m,n}}2$, where we have denoted by $\epsilon_{m,n}\in\set{1,3}$ the residue of $n-m$ modulo $4$. Because of this splitting of cases modulo $4$, we then choose further induction steps in such way that $n-m$ modulo $4$ is preserved, leading to the following configurations (see Figure~\ref{fig:uglystrategy} for the inductant relations between them):

\begin{figure}[tb]
    \centering
    \begin{tikzpicture}[n/.style = {draw, circle, inner sep=2pt}, a/.style = {-{Latex[scale=1]}, shorten <=2pt, shorten >=2pt}, scale=2.35]
        \node[n] (A0) at (0,0) {\familyname A0};
        \node[n] (A1) at (0,1) {\hatname A1};
        \node[n] (A2) at (0,2) {\hatname A2};
        \node[n] (A3) at (0,3) {\hatname A3};
        \node[n] (B1) at (1,1) {\hatname B1};
        \node[n] (B2) at (1,2) {\hatname B2};
        \node[n] (B3) at (1,3) {\hatname B3};
        \node[n] (C1) at (2,1) {\hatname C1};
        \node[n] (C2) at (2,2) {\hatname C2};
        \node[n] (D1) at (3,1) {\hatname D1};

        \draw[a] (A1)--(A0) node[midway, left=.3cm] {\tinylabel{\small$\zav{m, \frac{n-m-\epsilon_{m,n}}2}$}};
        \draw[a] (A2)--(A1) node[midway, left=.2cm] {\tinylabel{$(m,n-4)$}};
        \draw[a] (A3)--(A2) node[midway, left=.2cm] {\tinylabel{$(m,n-4)$}};
        \draw[a] (B2)--(B1) node[midway, left=.2cm] {\tinylabel{$(m,n-4)$}};
        \draw[a] (B3)--(B2) node[midway, left=.2cm] {\tinylabel{$(m,n-4)$}};
        \draw[a] (C2)--(C1) node[midway, left=.2cm] {\tinylabel{$(m,n-4)$}};

        \draw[a] (B1)--(A1) node[midway, below=.2cm] {\tinylabel{$(m-2,n-6)$}};
        \draw[a] (C1)--(B1) node[midway, below=.2cm] {\tinylabel{$(m-2,n-6)$}};
        \draw[a] (D1)--(C1) node[midway, below=.2cm] {\tinylabel{$(m-2,n-6)$}};
        \draw[a] (B2)--(A2) node[midway, below=.2cm] {\tinylabel{$(m-2,n-6)$}};
        \draw[a] (C2)--(B2) node[midway, below=.2cm] {\tinylabel{$(m-2,n-6)$}};
        \draw[a] (B3)--(A3) node[midway, below=.2cm] {\tinylabel{$(m-2,n-6)$}};
    \end{tikzpicture}
    \caption{Families of coordinate configurations in the inductive strategy for the ugly case. An arrow labeled with $(M,N)$ denotes that a family is an $(M,N)$-inductant of the family being pointed to.}
    \label{fig:uglystrategy}
\end{figure}
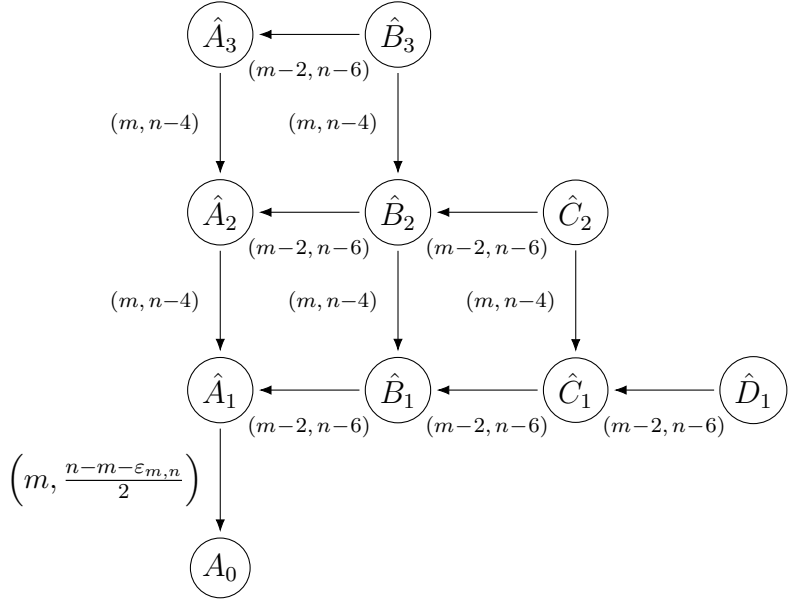

\begin{itemize}
\hatitem A1{
    For any ugly $(m,n)$ with $m\geq2$ and $n\geq3m-3$, we define $\hat A_1(m,n)$ with $k=1$ coordinate subvariety through
    \begin{align*}
        \tilde u_\emptyset &= m+1, & \tilde u_{\set1} &= 0,\\
        \tilde v_\emptyset &= \frac{n-m+2-\epsilon_{m,n}}{2}, & \tilde v_{\set1} &= \frac{n+m+\epsilon_{m,n}}{2},\\
        p_\emptyset &= \frac{mn+n+m^2+(1+\epsilon_{m,n})m+\epsilon_{m,n}-2}4, & p_{\set1} &= \ups\zav{m, {\textstyle\frac{n-m-\epsilon_{m,n}}{2}}} =\\&&&= \frac{mn+n-3m^2+(1-\epsilon_{m,n})m+8-\epsilon_{m,n}}4
    \end{align*}
    This is an $\zav{m,\frac{n-m-\epsilon_{m,n}}2}$-inductant of $A_0$ and comparing virtual dimensions in the ugly and nice cases of $A_0$, we get
    \begin{multline*}
        \vdim \hat A_1(m,n) = \vdim A_0(m,n) - \vdim A_0\zav{m,\frac{n-m-\epsilon_{m,n}}2} =\\
        =3\binom{m+1}3-\frac{n+m+1}2 -3\binom{m+1}3+\frac{n+m+2-\epsilon_{m,n}}2= \frac{1-\epsilon_{m,n}}{2},
    \end{multline*}
    i.e. $0$ for $n-m\equiv1\pmod4$ and $-1$ for $n-m\equiv3\pmod4$. In particular, $\hat A_1(m,n)$ is always superabundant and its virtual dimension only depends on $n-m\pmod4$.
}

\hatitem A2{
    For any ugly $(m,n)$ with $m\geq2$ and $n\geq 3m+1$, we define $\hat A_2(m,n)$ with $k=2$ coordinate subvarieties through
    \begin{align*}
        \tilde u_{\emptyset} &= m+1, &\tilde u_I &= 0\quad\text{for all other $I$,}\\
        \tilde v_\emptyset &= \frac{n-m-2-\epsilon_{m,n}}{2}, & \tilde v_{\set1} &= \frac{n+m-4+\epsilon_{m,n}}2,\\
        \tilde v_{\set2} &= \tilde v_{\set{1,2}} = 2, \span\omit\\
        p_\emptyset &= p_{\set1} = m+1, & p_{\set{2}} &= \frac{mn+n+m^2+(\epsilon_{m,n}-3)m+\epsilon_{m,n}-6}{4},\\
        &&p_{\set{1,2}} &= \ups\zav{m,\frac{n-m-4-\epsilon_{m,n}}{2}} =\\
        &&&= \frac{mn+n-3m^2-(3+\epsilon_{m,n})m+4-\epsilon_{m,n}}{4}.
    \end{align*}
    This is an $(m,n-4)$-inductant of $\hat A_1(m,n)$ and, since the virtual dimension of $\hat A_1(m,n)$ only depends on $n-m\pmod4$, we see that $\hat A_2(m,n)$ will always be equiabundant. Further, all inductants constructed from it will also be equiabundant.
}

\hatitem A3{
    For any ugly $(m,n)$ with $m\geq2$ and $n\geq 3m+5$, we define $\hat A_3(m,n)$ with $k=3$ coordinate subvarieties through
    \begin{align*}
        \tilde u_{\emptyset} &= m+1, &\tilde u_I &= 0\quad\text{for all other $I$,}\\
        \tilde v_\emptyset &= \frac{n-m-6-\epsilon_{m,n}}2, & \tilde v_{\set1} &= \frac{n+m-8+\epsilon_{m,n}}2,\\
        \tilde v_{\set2} &= \tilde v_{\set{1,2}} = \tilde v_{\set3} = \tilde v_{\set{1,3}} = 2, & \tilde v_{\set{2,3}} &= \tilde v_{\set{1,2,3}} = 0,\\
        p_{\set2} &= p_{\set{1,2}} = p_{\set3} = p_{\set{1,3}} = m+1, & p_I &= 0\quad\text{for all other $I$}.
    \end{align*}
    This is an $(m,n-4)$-inductant of $\hat A_2$; note that $p_{\set{2,3}} = p_{\set{1,2,3}} = 0$ is due to erasure of irrelevant points.
}

\hatitem B1{
    For any ugly $(m,n)$ with $m\geq4$ and $n\geq 3m-3$, we define $\hat B_1(m,n)$ with $k=2$ coordinate subvarieties through
    \begin{align*}
        \tilde u_\emptyset &= m-1, & \tilde u_{\set2} &= 2,\\
        \tilde u_{\set1} &= \tilde u_{\set{1,2}} = 0,\span\omit\\
        \tilde v_\emptyset &= \frac{n-m-2-\epsilon_{m,n}}2, & \tilde v_{\set1} &= \frac{n+m-8+\epsilon_{m,n}}2,\\
        \tilde v_{\set2} &= 2, & \tilde v_{\set{1,2}} &= 4,\\
        p_\emptyset &= \frac{n+5m+\epsilon_{m,n}-4}2, & p_{\set1} &= \frac{n-3m+4-\epsilon_{m,n}}2,\\
        p_{\set2} &= \frac{mn-n+m^2+(\epsilon_{m,n}-9)m+6-\epsilon_{m,n}}4, & p_{\set{1,2}} &= \ups\zav{m-2,\frac{n-m-4-\epsilon_{m,n}}2} =\\&&&= \frac{mn-n-3m^2+(7-\epsilon_{m,n})m+\epsilon_{m,n}}4.
    \end{align*}
    This is an $(m-2,n-6)$-inductant of $\hat A_1(m,n)$; in particular, since the virtual dimension of $\hat A_1(m,n)$ only depends on $n-m\pmod4$, this means $\hat B_1(m,n)$, as well as all families constructed from it as inductants, will always be equiabundant.
}

\hatitem B2{
    For any ugly $(m,n)$ with $m\geq4$ and $n\geq 3m+1$, we define $\hat B_2(m,n)$ with $k=3$ coordinate subvarieties through
    \begin{align*}
        \tilde u_\emptyset &= m-1, & \tilde u_{\set2} &= 2,\\
        \tilde u_{I} &= 0\quad\text{for all other $I$},\span\omit\\
        \tilde v_\emptyset &= \frac{n-m-6-\epsilon_{m,n}}2, & \tilde v_{\set1} &= \frac{n+m-12+\epsilon_{m,n}}2,\\
        \tilde v_{\set2} &= \tilde v_{\set3} = \tilde v_{\set{1,3}} = 2, & \tilde v_{\set{1,2}} &= 4,\\
        \tilde v_{\set{2,3}} &= \tilde v_{\set{1,2,3}} = 0,\span\omit\\
        p_\emptyset &= p_{\set1} = 2, & p_{\set2} &= p_{\set{1,2}} = m-1,\\
        p_{\set3} &= \frac{n+5m+\epsilon_{m,n}-8}2, & p_{\set{1,3}} &= \frac{n-3m-\epsilon_{m,n}}2,\\
        p_{\set{2,3}} &= p_{\set{1,2,3}} = 0.\span\omit
    \end{align*}
    This is an $(m,n-4)$-inductant of $\hat B_1$, and by Remark~\ref{rmrk:commuting-inductants}, also an $(m-2,n-6)$-inductant of $\hat A_2$. Note that $p_{\set{2,3}} = p_{\set{1,2,3}} = 0$ is due to erasure of irrelevant points.
}

\hatitem B3{
    For any ugly $(m,n)$ with $m\geq4$ and $n\geq 3m+5$, we define $\hat B_3(m,n)$ with $k=4$ subvarieties through
    \begin{align*}
        \tilde u_\emptyset &= m-1, & \tilde u_{\set2} &= 2,\\
        \tilde u_{I} &= 0\quad\text{for all other $I$},\span\omit\\
        \tilde v_\emptyset &= \frac{n-m-10-\epsilon_{m,n}}2, & \tilde v_{\set1} &= \frac{n+m-16+\epsilon_{m,n}}2,\\
        \tilde v_{\set2} &= \tilde v_{\set3} = \tilde v_{\set{1,3}} = \tilde v_{\set4} = \tilde v_{\set{1,4}} = 2,\span\omit\\
        \tilde v_{\set{1,2}} &= 4, & \tilde v_I &= 0\quad\text{for all other $I$,}\\
        p_{\set3} &= p_{\set{1,3}} = p_{\set4} = p_{\set{1,4}} = 2, & p_I &= 0\quad\text{for all other $I$.}
    \end{align*}
    This is an $(m,n-4)$-inductant of $\hat B_2$, and by Remark~\ref{rmrk:commuting-inductants}, also an $(m-2,n-6)$-inductant of $\hat A_3$. Note that $p_{\set{2,4}} = p_{\set{1,2,4}}=p_{\set{3,4}}=p_{1,3,4}=0$ is due to erasure of irrelevant points.
}

\hatitem C1{
    For any ugly $(m,n)$ with $m\geq6$ and $n\geq 3m-3$, we define $\hat C_1(m,n)$ with $k=3$ coordinate subvarieties through
    \begin{align*}
        \tilde u_\emptyset &= m-3, & \tilde u_{\set2} &= \tilde u_{\set3} = 2,\\
        \tilde u_I &= 0\quad\text{for all other $I$},\span\omit\\
        \tilde v_\emptyset &= \frac{n-m-6-\epsilon_{m,n}}2, & \tilde v_{\set1} &= \frac{n+m-16+\epsilon_{m,n}}{2},\\
        \tilde v_{\set2} &= \tilde v_{\set3} = 2, & \tilde v_{\set{1,2}} &= \tilde v_{\set{1,3}} = 4,\\
        \tilde v_{\set{2,3}} &= \tilde v_{\set{1,2,3}} = 0,\span\omit\\
        p_\emptyset &= 8, & p_{\set1} &= p_{\set{2,3}} = p_{\set{1,2,3}}=  0,\\
        p_{\set2} &= p_{\set3} = \frac{n+5m+\epsilon_{m,n}-20}2, & p_{\set{1,2}} &= p_{\set{1,3}} = \frac{n-3m+4-\epsilon_{m,n}}2.\\
    \end{align*}
    This is an $(m-2,n-6)$-inductant of $\hat B_1$. Note that $p_{\set{2,3}}=p_{\set{1,2,3}}=0$ is due to erasure of irrelevant points.
}

\hatitem C2{
    For any ugly $(m,n)$ with $m\geq6$ and $n\geq 3m+1$, we define $\hat C_2(m,n)$ with $k=4$ coordinate subvarieties through
    \begin{align*}
        \tilde u_\emptyset &= m-3, & \tilde u_{\set2} &= \tilde u_{\set3} = 2,\\
        \tilde u_I &= 0\quad\text{for all other $I$},\span\omit\\
        \tilde v_\emptyset &= \frac{n-m-10-\epsilon_{m,n}}2, & \tilde v_{\set1} &= \frac{n+m-20+\epsilon_{m,n}}{2},\\
        \tilde v_{\set2} &= \tilde v_{\set3} = \tilde v_{\set4} = \tilde v_{\set{1,4}} = 2, & \tilde v_{\set{1,2}} &= \tilde v_{\set{1,3}} = 4,\\
        \tilde v_I &= 0\quad\text{for all other $I$},\span\omit\\
        p_{\set2} &= p_{\set{3}} = p_{\set{1,2}} = p_{\set{1,3}} = 2, & p_{\set{4}} &= 8,\\
        p_I &= 0\quad\text{for all other $I$}.\span\omit
    \end{align*}
    This is an $(m,n-4)$-inductant of $\hat C_1$, and by Remark~\ref{rmrk:commuting-inductants}, also an $(m-2,n-6)$-inductant of $\hat B_1$. Note that $p_{\set{2,4}}=p_{\set{3,4}} = p_{\set{1,2,4}} = p_{\set{1,3,4}} = 0$ is due to erasure of irrelevant points.
}

\hatitem D1{
    For any ugly $(m,n)$ with $m\geq8$ and $n\geq 3m-3$, we define $\hat D_1(m,n)$ with $k=4$ coordinate subvarieties through
    \begin{align*}
        \tilde u_\emptyset &= m-5, & \tilde u_{\set2} &= \tilde u_{\set3} = \tilde u_{\set4} = 2,\\
        \tilde u_I &= 0\quad\text{for all other $I$},\span\omit\\
        \tilde v_\emptyset &= \frac{n-m-10-\epsilon_{m,n}}2, & \tilde v_{\set1} &= \frac{n+m-24+\epsilon_{m,n}}2,\\
        \tilde v_{\set2} &= \tilde v_{\set3} = \tilde v_{\set4} = 2, & \tilde v_{\set{1,2}} &= \tilde v_{\set{1,3}} = \tilde v_{\set{1,4}} = 4,\\
        \tilde v_I &= 0\quad\text{for all other $I$,}\span\omit\\
        p_{\set2} &= p_{\set3} = p_{\set4} = 8, & p_I &= 0\quad\text{for all other $I$.}
    \end{align*}
    This is an $(m-2,n-6)$-inductant of $\hat C_1$. Note that $p_{\set{2,4}} = p_{\set{3,4}} = p_{\set{1,2,4}} = p_{\set{1,3,4}} = 0$ is due to erasure of irrelevant points.
}
\end{itemize}
To see that each of these configurations satisfies the conditions \eqref{eq:coordsubvar-nonempty}, \eqref{eq:coordconfig-sanity}, note that the bounds listed for each are such that $\tilde u_\emptyset$ and $\tilde v_\emptyset$ will be positive, which we already saw is enough to ensure \eqref{eq:coordsubvar-nonempty}, \eqref{eq:coordconfig-sanity}.

\begin{prop}
    \label{prop:uglycertificates}
    The following are non-defective and superabundant (or even equiabundant):
    \begin{gather*}
        \hat A_1(4,9),\quad \hat A_1(4,11),\quad \hat A_2(2,7),\quad \hat A_2(2,9),\quad \hat A_3(2,11),\quad \hat A_3(2,13),\\
        \hat B_1(4,9),\quad \hat B_1(4,11),\quad \hat B_2(4,13),\quad \hat B_2(4,15),\quad \hat B_3(4,17),\quad \hat B_3(4,19),\\
        \hat C_1(6,15),\quad \hat C_1(6,17),\quad \hat C_2(6,19),\quad \hat C_2(6,21),\quad \hat D_1(8,21),\quad \hat D_1(8,23).
    \end{gather*}
\end{prop}
\begin{proof}
    Verified computationally (see Section~\ref{sec:computations}).
\end{proof}

As in the previous subsection, we go through the diagram of Figure~\ref{fig:uglystrategy} in reverse order by using induction steps facilitated by non-defectivity of inductants and the base cases of Proposition~\ref{prop:uglycertificates}

\begin{prop}
    \label{prop:D1hat}
    For any ugly $(m,n)$ with $m\geq 8$, $n\geq 3m-3$, the configuration $\hat D_1(m,n)$ is equiabundant and non-defective.
\end{prop}
\begin{proof}
Let $S$ be the $(m,n-4)$-inductant of $\hat D_1$, then it is easy to see that all points become irrelevant in this construction. Further, $S$ has $5$ coordinate subvarieties and $4$ of them do not share any equations between themselves. Hence any monomial that vanishes on all four of them must have degree at least $4$; in particular, no bidegree $(1,2)$ monomials vanish on them. Hence $S$ is always trivially equiabundant and non-defective, establishing the induction step
\[
    \text{$\hat D_1(m,n-4)$ non-defective} \implies \text{$\hat D_1(m,n)$ non-defective}.
\]

Similarly, let $S'$ be the $(m-2,n-6)$-inductant of $\hat D_1$. Through essentially the same argumentation, it is also trivially equiabundant and non-defective, which gives the induction
\[
    \text{$\hat D_1(m-2,n-6)$ non-defective} \implies \text{$\hat D_1(m,n)$ non-defective}.
\]

Applying these two inductions with the base cases $\hat D_1(8,21)$ and $\hat D_1(8,23)$, which we have by Proposition~\ref{prop:uglycertificates}, we obtain non-defectivity of all $\hat D_1(m,n)$ with $(m,n)$ ugly and $m\geq8$, $n\geq 3m-3$.
\end{proof}

\begin{prop}
    \label{prop:C2hat}
    For any ugly $(m,n)$ with $m\geq 6$, $n\geq 3m+1$, the configuration $\hat C_2(m,n)$ is equiabundant and non-defective.
\end{prop}
\begin{proof}
    The $(m,n-4)$-inductant of $\hat C_2$ will have all its points erased due to being irrelevant and it will admit no bidegree $(1,2)$ monomials vanishing on its subvarieties, since $4$ of them share no equations. This establishes the induction
    \[
    \text{$\hat C_2(m,n-4)$ non-defective} \implies \text{$\hat C_2(m,n)$ non-defective}.
    \]
    Further, the $(m-2,n-6)$-inductant of $\hat C_2$ is the same as the $(m,n-4)$-inductant of $\hat D_1$ due to Remark~\ref{rmrk:commuting-inductants}, which we already investigated in the proof of the previous proposition, hence establishing
    \[
    \text{$\hat C_2(m-2,n-6)$ non-defective} \implies \text{$\hat C_2(m,n)$ non-defective}.
    \]
    Applying these two to $\hat C_2(6,19)$ and $\hat C_2(6,21)$, which are non-defective by Proposition~\ref{prop:uglycertificates}, we obtain the desired result.
\end{proof}

\begin{prop}
    \label{prop:C1hat}
    For any ugly $(m,n)$ with $m\geq 6$, $n\geq 3m-3$, the configuration $\hat C_1(m,n)$ is equiabundant and non-defective.
\end{prop}
\begin{proof}
    By Proposition~\ref{prop:C2hat} and \ref{prop:D1hat}, we have induction steps
    \begin{align*}
        \text{$\hat C_1(m,n-4)$ non-defective} &\implies \text{$\hat C_1(m,n)$ non-defective}\quad\text{(for $m\geq 6$, $n\geq 3m+1$)},\\
        \text{$\hat C_1(m-2,n-6)$ non-defective} &\implies \text{$\hat C_1(m,n)$ non-defective}\quad\text{(for $m\geq 8$, $n\geq 3m-3$)},
    \end{align*}
    for ugly $(m,n)$.
    Applying these to base cases $\hat C_1(6,15)$ and $\hat C_1(6,17)$, which are non-defective by Proposition~\ref{prop:C1hat}, then obtains the desired result.
\end{proof}

\begin{prop}
    \label{prop:B3hat}
    For any ugly $(m,n)$ with $m\geq 4$, $n\geq 3m+5$, the configuration $\hat B_3(m,n)$ is equiabundant and non-defective.
\end{prop}
\begin{proof}
    Let $S$ be the $(m,n-4)$-inductant of $\hat B_3$. All of its points become irrelevant in the inductant construction. Further, its subvarieties admit no bidegree $(1,2)$ monomials to vanish on them, because $4$ of the $5$ subvarieties share no equations. Therefore $S$ is trivially equiabundant and non-defective, establishing the induction
    \[
    \text{$\hat B_3(m,n-4)$ non-defective} \implies \text{$\hat B_3(m,n)$ non-defective}.
    \]
    On the other hand, the $(m-2,n-6)$-inductant of $\hat B_3$ is the same as the $(m,n-4)$-inductant of $\hat C_2$ (by Remark~\ref{rmrk:commuting-inductants}) which we have already investigated, giving us the induction step
    \[
    \text{$\hat B_3(m-2,n-6)$ non-defective} \implies \text{$\hat B_3(m,n)$ non-defective}.
    \]
    Applying these induction steps to base cases $\hat B_3(4,17)$ and $\hat B_3(4,19)$, which are non-defective by Proposition~\ref{prop:uglycertificates}, then yields the desired result.
\end{proof}

\begin{prop}
    \label{prop:B2hat}
    For any ugly $(m,n)$ with $m\geq 4$, $n\geq 3m+1$, the configuration $\hat B_2(m,n)$ is equiabundant and non-defective.
\end{prop}
\begin{proof}
    By Propositions~\ref{prop:B3hat} and \ref{prop:C2hat}, we have induction steps
    \begin{align*}
        \text{$\hat B_2(m,n-4)$ non-defective} &\implies \text{$\hat B_2(m,n)$ non-defective}\quad\text{(for $m\geq 4$, $n\geq 3m+5$)},\\
        \text{$\hat B_2(m-2,n-6)$ non-defective} &\implies \text{$\hat B_2(m,n)$ non-defective}\quad\text{(for $m\geq 6$, $n\geq 3m+1$)},
    \end{align*}
    for ugly $(m,n)$.
    Applying these to base cases $\hat B_2(4,13)$ and $\hat B_2(4,15)$, which are non-defective by Proposition~\ref{prop:uglycertificates}, then yields the desired conclusion.
\end{proof}

\begin{prop}
    \label{prop:B1hat}
    For any ugly $(m,n)$ with $m\geq 4$, $n\geq 3m-3$, the configuration $\hat B_1(m,n)$ is equiabundant and non-defective.
\end{prop}
\begin{proof}
    By Propositions~\ref{prop:B2hat} and \ref{prop:C1hat}, we have induction steps
    \begin{align*}
        \text{$\hat B_1(m,n-4)$ non-defective} &\implies \text{$\hat B_1(m,n)$ non-defective}\quad\text{(for $m\geq 4$, $n\geq 3m+1$)},\\
        \text{$\hat B_1(m-2,n-6)$ non-defective} &\implies \text{$\hat B_1(m,n)$ non-defective}\quad\text{(for $m\geq 6$, $n\geq 3m-3$)},
    \end{align*}
    for ugly $(m,n)$
    Applying these to base cases $\hat B_1(4,9)$ and $\hat B_1(4,11)$, which are non-defective by Proposition~\ref{prop:uglycertificates}, then yields the desired result.
\end{proof}

\begin{prop}
    \label{prop:A3hat}
    For any ugly $(m,n)$ with $m\geq 2$, $n\geq 3m+5$, the configuration $\hat A_3(m,n)$ is equiabundant and non-defective.
\end{prop}
\begin{proof}
    Let $S$ be the $(m,n-4)$-inductant of $\hat A_3$, all its points become irrelevant in the inductant construction. Further, it admits no bidegree $(1,2)$ monomials to vanish on its subvarieties, since $4$ of the $5$ share no equations. Thus $S$ is always trivially equiabundant and non-defective, establishing the induction
    \[
    \text{$\hat A_3(m,n-4)$ non-defective} \implies \text{$\hat A_3(m,n)$ non-defective}.
    \]

    On the other hand, by Proposition~\ref{prop:B3hat}, we have the induction step
    \[
    \text{$\hat A_3(m-2,n-6)$ non-defective} \implies \text{$\hat A_3(m,n)$ non-defective}\quad\text{(for $m\geq 4$, $n\geq 3m+5$)}.
    \]
    for ugly $(m,n)$.
    Combining these two inductions with the bases cases $\hat A_3(2,11)$ and $\hat A_3(2,13)$, which are non-defective by Proposition~\ref{prop:uglycertificates}, then yields the desired conclusion.
\end{proof}

\begin{prop}
    \label{prop:A2hat}
    For any ugly $(m,n)$ with $m\geq 2$, $n\geq 3m+1$, the configuration $\hat A_2(m,n)$ is equiabundant and non-defective.
\end{prop}
\begin{proof}
    By Propositions~\ref{prop:A3hat} and \ref{prop:B2hat}, we have induction steps
    \begin{align*}
        \text{$\hat A_2(m,n-4)$ non-defective} &\implies \text{$\hat A_2(m,n)$ non-defective}\quad\text{(for $m\geq 2$, $n\geq 3m+5$)},\\
        \text{$\hat A_2(m-2,n-6)$ non-defective} &\implies \text{$\hat A_2(m,n)$ non-defective}\quad\text{(for $m\geq 4$, $n\geq 3m+1$)}.
    \end{align*}
    for ugly $(m,n)$.
    Applying these to base cases $\hat A_2(2,7)$ and $\hat A_2(2,9)$, which are non-defective by Proposition~\ref{prop:uglycertificates}, then yields the desired result.
\end{proof}

\begin{prop}
    \label{prop:A1hat}
    For any ugly $(m,n)$ with $m\geq 4$, $n\geq 3m-3$, the configuration $\hat A_1(m,n)$ is superabundant and non-defective.
\end{prop}
\begin{proof}
    By Propositions~\ref{prop:A2hat} and \ref{prop:B1hat}, we have induction steps
    \begin{align*}
        \text{$\hat A_1(m,n-4)$ non-defective} &\implies \text{$\hat A_1(m,n)$ non-defective}\quad\text{(for $m\geq 2$, $n\geq 3m+1$)},\\
        \text{$\hat A_1(m-2,n-6)$ non-defective} &\implies \text{$\hat A_1(m,n)$ non-defective}\quad\text{(for $m\geq 2$, $n\geq 3m-3$)},
    \end{align*}
    for ugly $(m,n)$.
    Applying these to base cases $\hat A_1(4,9)$ and $\hat A_1(4,11)$, which are non-defective by Proposition~\ref{prop:uglycertificates}, then yields the desired result.
\end{proof}

\begin{theorem}
    \label{thrm:A0ugly}
    For any ugly $(m,n)$ with $m\geq 4$, $n\geq m^3-m-3$, the configuration $A_0(m,n)$ is superabundant and non-defective.
\end{theorem}
\begin{proof}
    We will utilize Theorem~\ref{thrm:A0nice} together with the induction step
    \begin{multline}
        \label{eq:A0bigjump}
        \text{$A_0\zav{m,\frac{n-m-\epsilon_{m,n}}{2}}$ superabundant and non-defective} \implies\\\implies \text{$A_0(m,n)$ superabundant and non-defective}\quad\text{(for ugly $(m,n)$, $m\geq 4$, $n\geq 3m-3$)}
    \end{multline}
    that Proposition~\ref{prop:A1hat} provides. Note that $m^3-m-3>3m-3$, so this bound in the induction step is not an issue. Further, the bound  $\ceil{\frac{m^3-2m-5}2}$ from Theorem~\ref{thrm:A0nice} becomes $\frac{m^3-2m-4}2$ for even $m$, and under $n\geq m^3-m-3$, an easy splitting of cases based on $n-m\pmod4$ reveals that
    \[\frac{n-m-\epsilon_{m,n}}2 \geq \frac{m^3-2m-4}2.\]
    Hence for any $(m,n)$ satisfying the bounds of the present Theorem, we have that $A_0\zav{m,\frac{n-m-\epsilon_{m,n}}2}$ is superabundant and non-defective. Therefore, \eqref{eq:A0bigjump} shows that $A_0(m,n)$ is superabundant and non-defective.
\end{proof}

\section{Computations}
\label{sec:computations}

To verify non-defectivity of coordinate configurations, we adapt in Algorithm~\ref{alg:coordconfig} a standard Monte Carlo approach to checking defectivity of secant varieties, which we will now overview (cf. \cite{abo, abo-brambilla, abo-ottaviani-peterson, abo-vannieuwenhoven, torrance-vannieuwenhoven2021, torrance-vannieuwenhoven2022}).

Let $Z$ be a coordinate configuration and $\mathcal I_L$ be the ideal of polynomials that vanish on all the subvarieties of $Z$. We wish to find the dimension of $\mathcal I_Z(1,2)\subseteq\mathcal I_L(1,2)$, or equivalently its codimension. For this, we build a matrix whose entries are $\frac{\partial g}{\partial x_i}(P)$ or $\frac{\partial g}{\partial y_j}(P)$ as $g$ ranges over monomial generators of $\mathcal I_L(1,2)$ and $P$ ranges over the double points of $Z$. In practice, we know that many $\frac{\partial g}{\partial x_i}(P)$ or $\frac{\partial g}{\partial y_j}(P)$ will be zero for all $g$, so we only use those pairs $(P,x_i)$ or $(P,y_j)$ that do not give zeros identically. The rank of such a matrix is then the codimension of $\mathcal I_Z(1,2)$ in $\mathcal I_L(1,2)$.

The points $P$ of the configuration are chosen randomly, further, we constrain ourselves to integer-coordinate points and subsequently work over a suitable finite field $\ef_p$ -- we used $p=127$, the choice of a Mersenne prime allowing us to work directly with base-two representations when reducing integers modulo $p$. Working over a finite field suffices, because the rank of a matrix over a finite field is at most as large as its rank in characteristic $0$. Further, since we are looking for the dimension for a general choice of points, a random choice will usually be good enough. Because of the possibility of an unlucky choice of points, the result of a run of this algorithm can either be a definitive statement of non-defectivity or merely a probabilistic statement of defectivity.

In other words, an answer of \uv{defective} may be a false negative.
With a poor choice of parameters (for example, too small of a $p$), these false negatives could occur often. We had originally used $p=2^{13}-1 = 8191$ for the finite field, but later lowered this to $127$ in order to use $16$-bit integers instead of $32$-bit integers for the sake of saving memory.\footnote{To avoid integer overflows introducing arithmetic errors, one needs at least $\text{INTMAX}\geq (p-1)^2$, where $\text{INTMAX}$ is the maximal value that can be stored by the integer type used. Since we used signed $16$-bit integers, this was $\text{INTMAX} =2^{15}-1$ for us, making the use of $p=2^7-1$ viable.} In our experience, the false negatives were still infrequent enough, so we continued using $127$.

\begin{algorithm}
    \smallskip
    \raggedright
    \textsc{Input}: a coordinate configuration $Z$ with $k$ subvarieties in $\pe^{m,n}$ specified by $m$, $n$ and the parameters $\tilde u_I,\tilde v_I,p_I\in\zet_{\geq0}$ for $I\subseteq\set{1,\dots,k}$; a prime $p$\\
    \textsc{Output}: \uv{non-defective} or \uv{probably defective}
    \smallskip
    \begin{enumerate}[label={(\arabic*)}]
        \item Assign each subvariety $L_t$ a concrete set of equations $x_i=0$ and $y_j=0$ based on $\tilde u_I$, $\tilde v_I$.
        \item Compute the set $G$ of bidegree $(1,2)$ monomials in $\mathcal I_{\bigcup_{t\in\set{1,\dots,k}}L_t} = \bigcap_{t\in\set{1,\dots,k}}\mathcal I_{L_t}$.
        \item For each $I\subseteq\set{1,\dots,k}$, compute the list $W_I$ of variables that vanish on all $L_t$, $t\in I$.
        \item Initialize a $\abs G\times\zav{\sum_{I\subseteq\set{1,\dots,k}} p_I\cdot \abs{W_I}}$ matrix $M$ over $\ef_p$ filled with zeros; index the rows of $M$ with monomials from $G$.
        \item Initialize $\ell\gets 1$ and for each $I\subseteq\set{1,\dots,k}$, repeat the following $p_I$ times:
        \begin{itemize}[leftmargin=3em]
            \item[(5.1)] Choose a point $P=([x^{P}_0:\cdots:x^P_m], [y^P_0:\cdots:y^P_n])$ such that $x^P_i=0$ resp. $y^P_j=0$ whenever $x_i=0$ resp. $y_j=0$ is an equation of some $L_t$, $t\in I$ and the remaining coordinates are chosen randomly from $\set{0,1,\dots,p-1}$.
            \item[(5.2)] For each $z\in W_I$: write $M_{g,\ell} \gets \frac{\partial g}{\partial z}(P)$ for each $g\in G$, then increment $\ell\gets\ell+1$.
        \end{itemize}
        \item Compute $d:=\abs G - \rank M$ and $d_{\text{exp}}:=\max\set{0, \abs G - \sum_{I\subseteq \set{1,\dots,k}} p_I\cdot\min\set{m+n+1,\abs{W_I}}}$.
        \item If $d=d_{\text{exp}}$, return \uv{non-defective}, otherwise return \uv{probably defective}.
    \end{enumerate}
    \caption{Is a coordinate configuration $Z$ defective or not?}
    \label{alg:coordconfig}
\end{algorithm}

Implementationally, we were strongly inspired by \cite{torrance-vannieuwenhoven2021}. We implemented Algorithm~\ref{alg:coordconfig} in C++, compiling with the GNU Compiler Collection.\footnote{There are some mild formal differences between our implementation and the theory presented in this paper. Most notably, our implementation only allows bivariate polynomials in the role of parameter functions $p_I$, $\tilde u_I$, $\tilde v_I$ specifying families of coordinate configurations. Because of this, whenever a family uses a function that is a $d$-quasipolynomial and not a polynomial, we implement it internally as several separate families, one for each case modulo $d$. This happens with the families used in the ugly case.} The source code of this implementation may be found on GitHub, see \cite{gitrepo}.
For linear algebra over finite fields, we used
the Eigen \cite{eigen}
and FFLAS-FFPACK \cite{fflas-ffpack}
libraries; the latter depends
on Givaro \cite{givaro}
and on a BLAS implementation, for which we used OpenBLAS \cite{openblas}.
The building of the matrix $M$ in step (5) is parallelized with OpenMP directives.

We executed the code primarily on a virtual machine with 32 GB of main memory and 8 vCPUs. Later, we made some small modifications to the sets of certificates used in Propositions~\ref{prop:nicecertificates} and \ref{prop:uglycertificates} and therefore reran some of the smaller cases on a computer with 8 GB of main memory and an Intel Core i7-7500U CPU with a clock speed of 2.7 GHz.

Following \cite{torrance-vannieuwenhoven2021}, we generate certificates of the computations performed in the form of text files containing the basic information specifying the coordinate configuration in question, the particular points that were randomly chosen and rudimentary statistics about the size and length of the computation; this is of course dependent on hardware, among other factors. All of the certificates relevant to Propositions~\ref{prop:nicecertificates} and \ref{prop:uglycertificates} can be found at \cite{gitrepo} alongside the source code.
As an example, the following is a certificate for $B_0(2,4)$:
\medskip
\begin{verbatim}
Using random seed 1738187985 and the finite field with 127 elements
In the ambient space P^2 x P^4, computing the dimension of the O(1,2) part of
the ideal of a collection of subvarieties given by
{x_1 = x_2 = y_3 = y_4 = 0}
and double points supported at
([ 16 110  21], [ 71  96  82   5 125])
([79 92 78], [44 43 51 50 53])
([97 56 48], [57 82 74 19 50])
([38 31 85], [116  59  66  94  12])
([110   0   0], [ 15 106  78   0   0])
([120   0   0], [44 54 66  0  0])
([53  0  0], [70 66 97  0  0])
The 39 x 44 matrix was built in 0.0494281 s.
Computing the rank of the 39 x 44 matrix took 0.00149971 s.
The dimension is 0 vs. expected 0
The configuration is NON-DEFECTIVE (SUPERABUNDANT).
The entire computation took 0.0513569 s.
\end{verbatim}
\medskip
In terms of matrix size, $E_0(9,197)$ was the largest computation, involving a $67\,752 \times 67\,824$ matrix. Somewhat randomly, $E_0(9,196)$ actually took slightly longer in terms of time despite having a slightly smaller matrix, totaling $11\,210.9$ seconds, or approximately $3$~hours and $7$~minutes.

\def\arxiv#1{\href{https://arxiv.org/abs/#1}{arXiv:#1}}

\end{document}